\documentclass[12pt, reqno]{amsart}
\usepackage{amsmath}
\usepackage{amssymb,amsthm,amsmath}
\usepackage[numbers,sort&compress]{natbib}
\usepackage{amssymb,amsmath}
\usepackage{amsfonts}
\usepackage{mathrsfs}
\usepackage{latexsym}
\usepackage{amssymb}
\usepackage{amsthm}
\usepackage{indentfirst}
\usepackage{colortbl}
\date{
August 23, 2019}

\let\oldsection\section
\renewcommand\section{\setcounter{equation}{0}\oldsection}

\newtheorem{corollary}{Corollary}[section]
\newtheorem{theorem}{Theorem}[section]

\newtheorem{proposition}{Proposition}[section]

\newtheorem{remark}{Remark}[section]

\allowdisplaybreaks
\begin{document}

\title[Global strong solutions 1D compressible Navier-Stokes]{Global well-posedness of non-heat conductive compressible Navier-Stokes equations in 1D}

\author{Jinkai~Li}
\address[Jinkai~Li]{South China Research Center for Applied Mathematics and
                Interdisciplinary Studies, South China Normal University,
                Zhong Shan Avenue West 55, Tianhe District, Guangzhou
                510631, China}
\email{jklimath@m.scnu.edu.cn; jklimath@gmail.com}

%\author{Zhouping~Xin}
%\address[Zhouping~Xin]{The Institute of Mathematical Sciences,
%The Chinese University of Hong Kong, Hong Kong, P.R.China}
%\email{zpxin@ims.cuhk.edu.hk}


%%% ----------------------------------------------------------------------

\begin{abstract}
In this paper, the initial-boundary value problem of the 1D full compressible Navier-Stokes equations with positive constant viscosity but with zero heat conductivity is considered. Global well-posedness is established for any $H^1$ initial data.
The initial density is assumed only to be nonnegative, and, thus, is not necessary to be uniformly away from vacuum.
Comparing with the well-known result of Kazhikhov--Shelukhin
(Kazhikhov, A.~V.; Shelukhin, V.~V.: \emph{Unique global solution with
respect to time of initial boundary value problems for one-dimensional
equations of a viscous gas}, J.\,Appl.\,Math.\,Mech., \bf41 \rm(1977),
273--282.), the heat conductive coefficient is zero in this paper, and the initial vacuum is allowed.
\end{abstract}

%%% ----------------------------------------------------------------------
\maketitle

%\tableofcontents
\allowdisplaybreaks

\section{Introduction}
\subsection{The compressible Navier-Stokes equations}
The one-dimensional non-heat conductive
compressible Navier-Stokes equations read as
\begin{eqnarray}
   \rho_t+ (\rho u)_x&=&0,\label{1.1}\\
  \rho( u_t+u u_x)-\mu u_{xx}+p_x&=&0,
  \label{1.2}\\
  c_v\rho(\theta_t+u\theta_x)+u_xp&=&\mu(u_x)^2, \label{1.3}
\end{eqnarray}
where $\rho, u, \theta$, and $p$, respectively, denote the density,
velocity, absolute temperature, and pressure. The viscous coefficient
$\mu$ is assumed to be
a positive constant. The state equation for the ideal gas reads as
$p=R\rho\theta,$
where $R$ is a positive constant. Using the state equation, one can derive from (\ref{1.1}) and (\ref{1.3}) that
$$
   p_t+u p_x+\gamma u_xp=\mu(\gamma-1)( u_x)^2,
$$
where $\gamma-1=\frac{R}{c_v}.$ Therefore, we have the follow system
\begin{eqnarray}
  \rho_t+ (\rho u)_x&=&0,\label{eqrho-0}\\
  \rho( u_t+u u_x)-\mu u_{xx}+p_x&=&0,
  \label{equ-0}\\
  p_t+u p_x+\gamma u_xp&=&\mu(\gamma-1)( u_x)^2.\label{eqp-0}
\end{eqnarray}

The compressible Navier-Stokes equations have been extensively studied.
In the absence of vacuum, i.e., the case that the density
has a uniform positive lower bound, the local well-posedness was proved long time ago by Nash \cite{NASH62}, Itaya \cite{ITAYA71}, Vol'pert-Hudjaev \cite{VOLHUD72}, Tani \cite{TANI77}, Valli \cite{VALLI82}, and Lukaszewicz \cite{LUKAS84}; uniqueness was proved even earlier
by Graffi \cite{GRAFFI53} and Serrin \cite{SERRIN59}. Global well-posedness
of strong solutions in 1D has been
well-known since the works by Kanel \cite{KANEL68}, Kazhikhov--Shelukhin
\cite{KAZHIKOV77}, and Kazhikhov \cite{KAZHIKOV82};
global existence and uniqueness of weak
solutions was also established, see, e.g., Zlotnik--Amosov
\cite{ZLOAMO97,ZLOAMO98}, Chen--Hoff--Trivisa \cite{CHEHOFTRI00}, and Jiang--Zlotnik \cite{JIAZLO04}, and see Li--Liang \cite{LILIANG16} for the result on the large time behavior. The corresponding global
well-posedness results for the
multi-dimensional case were established only for
small perturbed initial data around some non-vacuum equilibrium or for
spherically symmetric large initial data, see, e.g., Matsumura--Nishida
\cite{MATNIS80,MATNIS81,MATNIS82,MATNIS83}, Ponce \cite{PONCE85},
Valli--Zajaczkowski \cite{VALZAJ86}, Deckelnick \cite{DECK92}, Jiang
\cite{JIANG96}, Hoff \cite{HOFF97}, Kobayashi--Shibata \cite{KOBSHI99},
Danchin \cite{DANCHI01}, Chen-Miao-Zhang \cite{CHENMIAOZHANG10}, Chikami--Danchin \cite{CHIDAN15}, Dachin-Xu \cite{DANXU18}, Fang-Zhang-Zi \cite{FZZ18}, and the
references therein.

In the presence of vacuum, that is the density may vanish
on some set or tends to zero at the far field, global existence of
weak solutions to the isentropic compressible Navier-Stokes equations
was first proved by Lions \cite{LIONS93,LIONS98},
with adiabatic constant $\gamma\geq\frac95$, and later generalized by
Feireisl--Novotn\'y--Petzeltov\'a \cite{FEIREISL01} to $\gamma>\frac32$,
and further by Jiang--Zhang \cite{JIAZHA03} to $\gamma>1$
for the axisymmetric solutions. For the full compressible Navier-Stokes
equations, global existence of the variational weak solutions was proved
by Feireisl \cite{FEIREISL04P,FEIREISL04B}, which however is not
applicable for the ideal gases. Local well-posedness of strong solutions
to the full compressible
Navier-Stokes equations, in the presence of vacuum, was proved by
Cho--Kim \cite{CHOKIM06-2}, see also Salvi--Stra$\check{\text s}$kraba
\cite{SALSTR93}, Cho--Choe--Kim \cite{CHOKIM04}, and Cho--Kim
\cite{CHOKIM06-1} for the isentropic case. Same to the non-vacuum case, the global well-posedness in 1D also
holds for the vacuum case, for arbitrary large initial data, see the recent work by the author \cite{JL1DHEAT}.
Generally, one can only
expect the solutions in
the homogeneous Sobolev spaces, see Li--Wang--Xin
\cite{LWX}. Global existence of strong solutions to the multi-dimensional
compressible Navier-Stokes equations, with small initial
data, in the presence of initial vacuum, was first proved by Huang--Li--Xin \cite{HLX12} for the isentropic case (see also Li--Xin
\cite{LIXIN13} for further developments), and later by
Huang--Li \cite{HUANGLI11} and Wen--Zhu \cite{WENZHU17} for the non-isentropic
case; in a recent work, the author \cite{JL3DHEATSMALL} proved the global well-posedness result under the assumption that some scaling invariant quantity is small. Due to the finite blow-up results in \cite{XIN98,XINYAN13},
the global solutions
obtained in \cite{HUANGLI11,WENZHU17,JL3DHEATSMALL} must have
unbounded entropy if the initial density is compactly supported; however, if the initial density has vacuum at the far field only, one can expect the global entropy-bounded solutions, see the recent work by the author and Xin \cite{LIXIN17,LIXIN19}.

%  \textcolor{red}{
%  An interesting problem is to show the positivity of the temperature to the full compressible Navier-Stokes equations. Along this direction, some
%  results have been proved by Mellet--Vasseur \cite{MV09} and Baer--Vasseur \cite{BV13}, where
%  under some appropriate constitutive assumptions on the internal energy, the pressure law, the viscous coefficients, and the heat conductive
%  coefficient, they have shown that the temperature has a uniform positive lower bound, with the bound depending on the time,
%  at each time after the initial time,
%  for the initial-boundary value problem, under the non heat flux boundary condition. However, due to the constitutive assumptions in \cite{MV09,BV13}, the results there do not apply to the ideal gas with constant viscous and heat conductive coefficients.
%  Besides, the
%  positivity of the temperature inside the space time domain to the system considered in \cite{MV09,BV13} is also not clear if replacing the non heat flux temperature boundary condition with the zero temperature boundary condition.}

In all the global well-posedness results \cite{KAZHIKOV77,KAZHIKOV82,ZLOAMO97,ZLOAMO98,CHEHOFTRI00,JIAZLO04,LILIANG16},
for the heat conductive compressible Navier-Stokes equations in 1D,
the density was assumed uniformly away from vacuum.
For the vacuum case, global well-posedness of heat conductive compressible Navier-Stokes equations in 1D was proved by Wen-Zhu \cite{WENZHU13} with the heat conductive coefficient $\kappa\approx 1+\theta^q$, for
positive $q$ suitably large, and by the author \cite{JL1DHEAT} with $\kappa\equiv Const.$

The aim of this paper is to study the global well-posedness of strong
solutions to the one-dimensional non-heat conductive compressible Navier-Stokes equations, i.e., system (\ref{1.1})--(\ref{1.3}), with constant
viscosity, in the presence of vacuum; this is the counterpart of the paper \cite{JL1DHEAT} where the heat conductive case was considered. To our best knowledge, global well-posedness of 1D non-heat conductive compressible Navier-Stokes equations for arbitrary large initial data is not known before, no matter the vacuum is contained or not.

The results of this paper will be proven
in the Lagrangian flow map coordinate being stated in the next subsection; however, it can be
equivalently translated back to the corresponding one in the
Euler coordinate.

\subsection{The Lagrangian coordinates and main result} Let
$\eta(y,t)$ be the flow map governed by $u$, that is
\begin{equation*}\label{flowmap}
  \left\{
  \begin{array}{l}
   \eta_t(y,t)=u(\eta(y,t),t),\\
  \eta(y,0)=y.
  \end{array}
  \right.
\end{equation*}
Denote by $\varrho, v$, and $\pi$ the density, velocity, and pressure, respectively, in the Lagrangian coordinate,
that is
\begin{eqnarray*}
  \varrho(y,t):=\rho(\eta(y,t),t),\quad v(y,t):=u(\eta(y,t),t), \quad \pi(y,t):=p(\eta(y,t),t),
\end{eqnarray*}
and introduce a function $J=J(y,t)=\eta_y(y,t)$.
Then, it follows
\begin{equation}
\label{eqJ}
J_t =v_y,
\end{equation}
and system (\ref{eqrho-0})--(\ref{eqp-0}) can be rewritten
in the Lagrangian coordinate as
\begin{eqnarray}
   \varrho_t+\frac{ v_y}{J}\varrho&=&0,\label{eqrho}\\
  \varrho v_t-\frac{\mu}{J} \left(\frac{ v_y}{J}
  \right)_y+\frac{ \pi_y}{J}&=&0,\label{eqv}\\
   \pi_t+\gamma\frac{ v_y}{J}\pi
  &=&\mu(\gamma-1)\left(\frac{ v_y}{J}\right)^2.\label{eqtheta}
\end{eqnarray}

Due to (\ref{eqJ}) and (\ref{eqrho}), it is straightforward that
$$
 (J\varrho)_t= J_t\varrho+J \varrho_t=
 v_y\varrho-J\frac{ v_y}{J}\varrho=0,
$$
from which, by setting $\varrho|_{t=0}=\varrho_0$ and noticing that
$J|_{t=0}=1$, we have $J\varrho=\varrho_0.$
Therefore, one can replace (\ref{eqrho}) with (\ref{eqJ}), by setting
$\varrho=\frac{\varrho_0}{J}$, and rewrite (\ref{eqv}) as
\begin{equation*}
  \varrho_0 v_t-\mu \left(\frac{ v_y}{J}\right)_y
  + \pi_y=0.
\end{equation*}

In summary, we only need to consider the following system
\begin{eqnarray}
   J_t&=& v_y,\label{EQJ}\\
  \varrho_0 v_t-\mu \left(\frac{ v_y}{J}\right)_y
  + \pi_y&=&0,\label{EQv}\\
 \pi_t+\gamma\frac{ v_y}{J}\pi
  &=&\mu(\gamma-1)\left(\frac{  v_y}{J}\right)^2.\label{EQpi}
\end{eqnarray}

We consider the initial-boundary value problem on the interval $(0,L)$, with $L>0$, and the
boundary and initial conditions read as
\begin{equation}
  \label{BC}
  v(0,t)=v(L,t)=0
\end{equation}
and
\begin{equation}
  \label{IC}
  (J,\varrho_0v,\pi)|_{t=0}=(1,\varrho_0v_0,\pi_0).
\end{equation}
We point out that here we put the initial condition on $\varrho_0v$ rather than on $v$. As will be shown in Theorem \ref{thm},
in the below, we can guarantee the continuity in time of $\varrho_0v$ but not necessary of $v$, if the initial data lies only in $H^1$.

For $1\leq q\leq\infty$ and positive integer $m$, we use $L^q=L^q((0,L))$ and $W^{m,q}=W^{m,q}((0,L))$ to denote the standard Lebesgue and Sobolev spaces, respectively, and in the case that $q=2$, we use $H^m$
instead of $W^{m,2}$. $H_0^1$ consists of all functions $v\in H^1$ satisfying $v(0)=v(L)=0$.
We always use $\|u\|_q$ to denote the $L^q$ norm of $u$.

The main result of this paper is the following:

\begin{theorem}
\label{thm}
  Assume $0\leq\varrho_0,\pi_0\in L^\infty,$ and $v_0\in H_0^1.$ Then, there is a unique global solution $(J, v, \pi)$ to system (\ref{EQJ})--(\ref{EQpi}), subject to (\ref{BC})--(\ref{IC}), satisfying
  \begin{eqnarray*}
    0<J\in C([0,T]; H^1),&& J_t\in L^\infty(0,T;L^2)\cap L^2(0,T; H^1),\\
    \varrho_0v\in C([0,T];L^2),&& v\in L^\infty(0,T;H^1)\cap L^2(0,T; H^2),\\
    \sqrt{\varrho_0}v_t\in L^2(0,T; L^2),&&\sqrt tv_t\in L^2(0,T;H^1),\\
    0\leq\pi\in C([0,T]; H^1),&&\pi_t\in L^{\frac43}(0,T; H^1),
  \end{eqnarray*}
  for any $T\in(0,\infty)$.
\end{theorem}

\begin{remark}
  The arguments presented in this paper also
  work for the free boundary value
  problem in which the boundary condition for $v$ in (\ref{BC}) is replaced by
  $$
  \left(\mu\frac{ v_y}{J}-\pi\right)\Big|_{y=0,L}=0.
  $$
  In fact, all the energy estimates obtained in this paper hold if replacing the boundary condition (\ref{BC}) with the above one, by copying or slightly modifying the proof.
\end{remark}

Throughout this paper, we use $C$ to denote a general positive constant which may different from line to line.

\section{Local and global well-posedness: without vacuum}
\label{SecApri}
This section is devoted to establishing the global well-posedness in the absence of vacuum which will be the base to
prove the corresponding result in the presence of vacuum in the next section.

We start with the following local existence result of which the proof will be given in the appendix.

\begin{proposition}
  \label{PropLocal}
Assume that $(\varrho_0, J_0, v_0, \pi_0)$ satisfies
\begin{eqnarray*}
  &0<\underline\varrho\leq\varrho_0\leq\bar\varrho<\infty,\quad 0<\underline J\leq J_0\leq\bar J<\infty,\\
   &\pi_0\geq0,\quad (\varrho_0, J_0, \pi_0)\in H^1,\quad v_0\in H_0^1,
\end{eqnarray*}
for positive numbers $\underline\varrho, \bar\varrho, \underline J,$ and $\bar J$.

Then, there is a positive time $T_0$ depending only on
$R$, $\gamma$, $\mu$, $\underline\varrho$, $\bar\varrho$, $\underline J$, $\overline J,$ and $\|(J_0, v_0, \pi_0)\|_{H^1}$, such that system (\ref{EQJ})--(\ref{EQpi}), subject to (\ref{BC})--(\ref{IC}), has a unique solution $(J,v,\pi)$ on $(0,L)\times(0,T_0)$, satisfying
 \begin{eqnarray*}
&0<J\in C([0,T_0]; H^1),\quad  J_t\in L^\infty(0,T_0); L^2),\\
&v\in C([0,T_0]; H^1_0)\cap L^2(0,T_0; H^2), \quad  v_t \in L^2(0,T_0; L^2),\\
&0\leq\pi\in C([0,T_0]; H^1), \quad \pi_t\in L^\infty(0,T_0; L^2).
  \end{eqnarray*}
\end{proposition}

In the rest of this section, we always
assume that $(J,v,\pi)$ is a solution to system (\ref{EQJ})--(\ref{EQpi}), subject to (\ref{BC})--(\ref{IC}),
on $(0,L)\times(0,T)$, satisfying the regularities stated in Proposition \ref{PropLocal}, with $T_0$ there replaced by
some positive time $T$. A series of a priori
estimates of $(J,v,\pi)$, independent of the lower bound
of the density, are carried out in this section.

We start with the basic energy identity.

\begin{proposition}
\label{PROPBASIC}
It holds that
\begin{equation*}
  \int_0^LJ(y,t)dy=\ell_0
\end{equation*}
and
$$
\left(\int_0^L\left(\frac{\varrho_0}{2}v^2+\frac{J\pi}{\gamma-1}\right)dy
\right)(t)=E_0,
$$
for any $t\in(0,\infty)$, where $\ell_0:=\int_0^LJ_0dy$ and $E_0:=
\int_0^L\left(\frac{\varrho_0}{2}v_0^2+\frac{\pi_0}{\gamma-1}\right)dy$.
\end{proposition}

\begin{proof}
  The first conclusion follows directly from integrating (\ref{EQJ})
  with respect to $y$ over $(0,L)$ and using the boundary condition (\ref{BC}). Multiplying
  equation (\ref{EQv}) by $v$, integrating the resultant over $(0,L)$,
  one gets from integrating by parts that
  $$
  \frac12\frac{d}{dt}\int_0^L\varrho_0v^2dy+\mu\int_0^L\frac{( v_y)^2}{J}dy =\int_0^L v_y\pi dy.
  $$
  Multiplying (\ref{EQpi}) with $J$ and integrating the resultant over $(0,L)$, it follows from (\ref{EQJ}) that
  $$
  \frac{d}{dt}\int_0^LJ\pi dy+(\gamma-1)\int_0^L v_y\pi dy
  =\mu(\gamma-1)\int_0^L\frac{( v_y)^2}{J}dy,
  $$
  which, combined with the previous equality, leads to
  $$
  \frac{d}{dt}\int_0^L\left(\frac{\varrho_0}{2}v^2+\frac{J\pi}{\gamma-1}
  \right)dy
  =0,
  $$
  the second conclusion follows.
\end{proof}

Next, we carry out the estimate on the lower bound of $J$. To this end, we perform some calculations in
the spirit of \cite{KAZHIKOV77} as preparations.

Due to (\ref{EQJ}), it follows from (\ref{EQv}) that
$$
\varrho_0 v_t-\mu(\log J)_{yt}+\pi_y=0.
$$
Integrating the above equation with respect to $t$ over $(0,t)$ yields
$$
\varrho_0(v-v_0)-\mu (\log J-\log J_0)_y+  \left(\int_0^t\pi d\tau\right)_y =0,
$$
from which, integrating with respect to $y$ over $(z,y)$, one obtains
\begin{eqnarray*}
  \int_z^y\varrho_0(v-v_0)d\xi -\mu\left(\log \frac{J}{J_0}(y,t)-\log \frac{J}{J_0}(z,t) \right)&&\\
   +\int_0^t(\pi(y,\tau)-\pi(z,\tau))d\tau&=&0,\quad\forall y,z\in(0,L).
\end{eqnarray*}
Thanks to this, noticing that
$$
\int_z^y\varrho_0(v-v_0)d\xi= \int_0^y\varrho_0(v-v_0)d\xi- \int_0^z\varrho_0(v-v_0)d\xi,
$$
and rearranging the terms, one obtains
\begin{eqnarray*}
  &&\int_0^y
  \varrho_0(v-v_0)d\xi-\mu\log \frac{J}{J_0}(y,t)+\int_0^t\pi(y,\tau)d\tau \\
  &=&\int_0^z
  \varrho_0(v-v_0)d\xi-\mu \log \frac{J}{J_0}(z,t)+\int_0^t\pi(z,\tau)d\tau,\qquad
  \forall y,z\in(0,L).
\end{eqnarray*}
Therefore, both sides of the above equality are independent of the spacial variable, that is
\begin{equation*}
  \int_0^y
  \varrho_0(v-v_0)d\xi-\mu\log \frac{J}{J_0}+\int_0^t\pi d\tau=h(t),
\end{equation*}
for some function $h$, from which, one can easily get
\begin{equation}
  \frac{J}{J_0}HB=e^{\frac1\mu\int_0^t\pi d\tau},\label{N1}
\end{equation}
where
$$
H=H(t)=e^{\frac{h(t)}{\mu}},\quad\mbox{and}\quad B=B(y,t)=e^{\frac1\mu\int_0^y\varrho_0(v_0-v)d\xi}.
$$
Multiplying both sides of the above with $\pi$ leads to
$$
\frac{HB}{\mu J_0}J\pi= \left(e^{\frac1\mu\int_0^t\pi d\tau}\right)_t,
$$
from which, integrating with respect to $t$, one arrives at
$$
e^{\frac1\mu\int_0^t\pi d\tau}=1+\frac{1}{\mu J_0}\int_0^tBHJ\pi d\tau.
$$
Thanks to the above, one can obtain from (\ref{N1}) that
\begin{equation}
  \label{N3}
  JHB=J_0+\frac1\mu\int_0^tHBJ\pi d\tau.
\end{equation}

A prior positive lower bound of $J$ is stated in the
following proposition:

\begin{proposition}
The following estimate holds
  \label{PROPEstJ}
$$
J\geq \underline J\exp\left\{-\frac4\mu\sqrt{2m_0E_0}-\frac{(\gamma-1)E_0}{\mu\ell_0}e^{\frac4\mu\sqrt{2m_0E_0}}t\right\},
$$
for any $t\in[0,\infty)$.
\end{proposition}

\begin{proof}
By Proposition \ref{PROPBASIC}, it follows from the H\"older inequality
that
\begin{eqnarray*}
  \left|\int_0^y\varrho_0(v-v_0)d\xi\right|\leq\int_0^L(|\varrho_0v|+|
  \varrho_0v_0|)d\xi
  \leq 2\sqrt{2m_0E_0},
\end{eqnarray*}
where $m_0=\int_0^L\varrho_0 dy$, and, thus,
\begin{equation}
\label{EstB}
e^{-\frac2\mu\sqrt{2m_0E_0}}
\leq B(y,t)\leq e^{\frac2\mu\sqrt{2m_0E_0}}.
\end{equation}
Applying Proposition \ref{PROPBASIC}, using (\ref{EstB}), and integrating (\ref{N3}) over $(0,L)$, one deduces
\begin{eqnarray*}
  \ell_0H(t)&=&\int_0^LJHdy\leq e^{\frac2\mu\sqrt{2m_0E_0}}\int_0^LJHBdy \\
  &=&e^{\frac2\mu\sqrt{2m_0E_0}}\left[\ell_0+\frac1\mu\int_0^tH\left(\int_0^LBJ\pi dy\right)d\tau\right]\\
  &\leq&e^{\frac2\mu\sqrt{2m_0E_0}}\left(\ell_0+\frac{(\gamma-1)E_0}{\mu}e^{\frac2\mu\sqrt{2m_0E_0}}\int_0^tHd\tau\right),
\end{eqnarray*}
and, thus,
$$
H(t)\leq e^{\frac2\mu\sqrt{2m_0E_0}}\left(1+\frac{(\gamma-1)E_0}{\mu\ell_0}e^{\frac2\mu\sqrt{2m_0E_0}}\int_0^tHd\tau\right).
$$
Applying the Gronwall inequality to the above yields
\begin{equation*}
  H(t)\leq\exp\left\{\frac2\mu\sqrt{2m_0E_0}+\frac{(\gamma-1)E_0}{\mu\ell_0}e^{\frac4\mu\sqrt{2m_0E_0}}t\right\}.
\end{equation*}
With the aid of this and recalling $\pi\geq0$ and (\ref{EstB}), one obtains from (\ref{N1}) that
\begin{eqnarray*}
  J&=&H^{-1}B^{-1}J_0e^{\frac1\mu\int_0^t\pi d\tau}\geq H^{-1}B^{-1}\underline J\\
  &\geq&\underline J\exp\left\{-\frac4\mu\sqrt{2m_0E_0}-\frac{(\gamma-1)E_0}{\mu\ell_0}e^{\frac4\mu\sqrt{2m_0E_0}}t\right\},
\end{eqnarray*}
the conclusion follows.
\end{proof}

Before continuing the argument, let us introduce the key quantity of this paper, the effective viscous flux $G$, defined as
\begin{equation}
  G:=\mu\frac{v_y}{J}-\pi. \label{G}
\end{equation}
By some straightforward calculations, one can easily derive the equation for $G$ from (\ref{EQJ})--(\ref{EQpi}) as
\begin{equation}
\label{EQG}
   G_t-\frac\mu J \left(\frac{ G_y}{\varrho_0}\right)_y=-\gamma\frac{ v_y}{J}G.
\end{equation}
Moreover, noticing that $\varrho_0 v_t= G_y$, it is clear from the boundary condition of $v$, i.e., (\ref{BC}), that
\begin{equation}
   G_y|_{y=0,L}=0. \label{BCG}
\end{equation}

The next proposition concerning the estimate on $G$ is the key of proving the $H^1$ estimates on $(J, v,
\pi)$ later.

\begin{proposition}
\label{PropEstG}
The following estimate holds
$$
\sup_{0\leq t\leq T}\|G\|_2^2+\int_0^T\left(\|G\|_\infty^4+\left\|\frac{ G_y}{\sqrt{\varrho_0}}\right\|_2^2\right)dt\leq C\|G_0\|_2^2,
$$
where $G_0=\mu\frac{v_0'}{J_0}-\pi_0$ and $C$ depends only on $\gamma, \mu, \bar\varrho, \ell_0, \underline J, m_0, E_0,$ and $T$.
\end{proposition}

\begin{proof}
  Multiplying (\ref{EQG}) with $JG$ and recalling the boundary condition (\ref{BCG}),
  it follows from integration by parts and (\ref{EQJ}) that
  \begin{equation}
    \frac12\frac{d}{dt}\|\sqrt JG\|_2^2+\mu\left\|\frac{ G_y}{\sqrt{\varrho_0}}\right\|_2^2
    =\left(\frac12-\gamma\right)\int_0^Lv_yG^2dy. \label{N4}
  \end{equation}
  Integration by parts and the H\"older inequality yield
  \begin{eqnarray*}
    \left|\int_0^Lv_yG^2dy\right|&=&2\left|\int_0^LvGG_ydy\right|\leq 2\left\|\frac{G_y}{\sqrt{\varrho_0}}\right\|_2
    \|\sqrt{\varrho_0}v\|_2\|G\|_\infty.
  \end{eqnarray*}
  By the Gagliardo-Nirenberg inequality and applying Proposition \ref{PROPEstJ}, it follows
  \begin{eqnarray}
    \|G\|_\infty&\leq& C\|G\|_2^{\frac12}\|G\|_{H^1}^{\frac12}\leq C\left(\|G\|_2+\|G\|_2^{\frac12}\left\|\frac{G_y}{\sqrt{\varrho_0}}\right\|_2^{\frac12}\right)\nonumber\\
    &\leq&C\left(\|\sqrt JG\|_2+\|\sqrt JG\|_2^{\frac12}\left\|\frac{G_y}{\sqrt{\varrho_0}}\right\|_2^{\frac12}\right).
    \label{N3-1}
  \end{eqnarray}
  Combining the previous two inequalities, it follows from the Young inequality and Proposition \ref{PROPBASIC} that
  \begin{eqnarray*}
    \left|\int_0^Lv_yG^2dy\right|&\leq&C\left\|\frac{G_y}{\sqrt{\varrho_0}}\right\|_2
    \|\sqrt{\varrho_0}v\|_2\left(\|\sqrt JG\|_2+\|\sqrt JG\|_2^{\frac12}\left\|\frac{G_y}{\sqrt{\varrho_0}}\right\|_2^{\frac12}\right)\\
    &\leq&\varepsilon\left\|\frac{G_y}{\sqrt{\varrho_0}}\right\|_2^2+C_\varepsilon(E_0+E_0^2)\|\sqrt JG\|_2^2,
  \end{eqnarray*}
  for any positive $\varepsilon$. Substituting the above into (\ref{N4}) with suitably chosen $\varepsilon$, one obtains
  \begin{equation*}
   \frac{d}{dt}\|\sqrt JG\|_2^2+\mu\left\|\frac{G_y}{\sqrt{\varrho_0}}\right\|_2^2
   \leq C\|\sqrt JG\|_2^2,
  \end{equation*}
  which leads to the conclusion by applying the Gronwall inequality and simply using (\ref{N3-1}) and Proposition \ref{PROPEstJ}.
\end{proof}

The uniform upper bounds of $J, \pi$ can now be proved as in the next proposition.

\begin{proposition}\label{PropJPiInfty}
The following estimate holds
$$
\sup_{0\leq t\leq T}(\|\pi\|_\infty+\|J\|_\infty)\leq C(1+\bar J+\|\pi_0\|_\infty),
$$
for a positive constant $C$ depending only on $\gamma, \mu, \bar\varrho, \ell_0, \underline J, m_0, E_0, \|G_0\|_2$, and $T$.
\end{proposition}

\begin{proof}
Noticing that $v_y=\frac J\mu(G+\pi)$, one can rewrite (\ref{EQpi}) as
\begin{equation}
 \pi_t+\frac{\pi^2}{\mu}=\frac{\gamma-1}{\mu}G^2+\frac{\gamma-2}{\mu}G\pi,\label{EQpi-v0}
\end{equation}
from which one can further derive
\begin{equation}
   \pi_t+\frac1\mu\left(\pi-\frac{\gamma-2}{2}G\right)^2=\frac{\gamma^2}{4\mu}G^2. \label{EQpi-v1}
\end{equation}
  The estimate for $\pi$ follows straightforwardly from integrating (\ref{EQpi-v1}) with respect to $t$ and applying Proposition \ref{PropEstG}. As for the estimate for $J$, noticing that (\ref{EQJ}) can be rewritten in terms of $G$ and $\pi$ as
  $ J_t=\frac J\mu(G+\pi),$ the conclusion follows from the Gronwall inequality by
  Proposition \ref{PROPEstJ}, Proposition \ref{PropEstG}, and the estimate for $\pi$ just proved.
\end{proof}

A priori $L^\infty(0,T; H^1)$ estimate for $(J,\pi)$ is given in the next proposition.

\begin{proposition}
\label{PropJPiy}
The following estimate holds
\begin{equation*}
  \sup_{0\leq t\leq T}(\|J_y\|_2+\|\pi_y\|_2)\leq C(1+\|J_0'\|_2+\|\pi_0'\|_2),
\end{equation*}
for a positive constant $C$ depending only on $\gamma, \mu,\bar\varrho,  \ell_0, \underline J, \bar J, m_0, \|\pi_0\|_\infty,
E_0, \|G_0\|_2$, and $T$.
\end{proposition}

\begin{proof}
  Differentiating (\ref{EQpi-v0}) with respect to $y$ gives
  $$
  \partial_t\pi_y+\frac2\mu\pi\pi_y=\frac{2(\gamma-1)}{\mu}GG_y+\frac{\gamma-2}{\mu}(\pi_y G+\pi G_y).
  $$
  Multiplying the above equation with $\pi_y$ and integrating over $(0,L)$, one deduces
  \begin{eqnarray*}
    &&\frac12\frac{d}{dt}\|\pi_y\|_2^2+\frac2\mu\int_0^L\pi|\pi_y|^2dy\\
    &=&\frac{2(\gamma-1)}{\mu}\int_0^LGG_y\pi_ydy+\frac{\gamma-2}{\mu}
    \int_0^L(G|\pi_y|^2+\pi G_y\pi_y)dy \\
    &\leq&C\|G_y\|_2^2+C(\|G\|_\infty^2+1+\|\pi\|_\infty^2)\|\pi_y\|_2^2,
  \end{eqnarray*}
  and, thus, by the Gronwall inequality, and applying Proposition \ref{PropEstG} and Proposition \ref{PropJPiInfty}, one gets
  \begin{eqnarray*}
    \sup_{0\leq t\leq T}\|\pi_y\|_2^2&\leq& e^{C\int_0^T(1+\|G\|_\infty^2+\|\pi\|_\infty^2)dt}\left(\|\pi_0'\|_2^2+C\int_0^T\|G_y\|_2^2dt\right)\\
    &\leq& C(1+\|\pi_0'\|_2^2).
  \end{eqnarray*}
  Note that
  \begin{eqnarray*}
    (\log J)_{yt}=\left(\frac{J_t}{J}\right)_y=\left(\frac{v_y}{J}\right)_y
    =\frac1\mu(G+\pi)_y.
  \end{eqnarray*}
  Therefore, by Proposition \ref{PropEstG} and the estimate just obtained for $\|\pi_y\|_2$, it follows that
  \begin{eqnarray*}
    \sup_{0\leq t\leq T}\|(\log J)_y\|_2&=&\sup_{0\leq t\leq T}\left\|(\log J_0)'+\int_0^t(\log J)_{yt}d\tau\right\|_2\\
    &\leq&\left\|\frac{J_0'}{J_0}\right\|_2+\int_0^T\|(\log J)_{yt}\|_2d\tau\\
    &\leq&\frac{\|J_0'\|_2}{\underline J}+\frac1\mu\int_0^T(\|G_y\|_2+\|\pi_y\|_2)d\tau\\
    &\leq& C(1+\|J_0'\|_2),
  \end{eqnarray*}
  and further by Proposition \ref{PropJPiInfty} that
  \begin{equation*}
    \sup_{0\leq t\leq T}\|J_y\|_2=\sup_{0\leq t\leq T}\|J(\log J)_y\|_2\leq C(1+\|J_0'\|_2),
  \end{equation*}
  proving the conclusion.
\end{proof}

\begin{corollary}
  \label{CorVy}
It holds that
$$
\sup_{0\leq t\leq T}\|(J_t,v_y)\|_2^2+\int_0^T\left(\|(\sqrt{\varrho_0}v_t,v_{yy},J_{yt})\|_2^2+
\|\pi_t\|_\infty^2+\|\pi_{yt}\|_2^{\frac43}\right)dt\leq C,
$$
for a positive constant $C$ depending only on $\gamma, \mu, \bar\varrho, \ell_0, \underline J, \bar J, m_0, \|\pi_0\|_\infty, E_0$, $\|G_0\|_2$, $\|J_0'\|_2$, $\|\pi_0'\|_2$, and $T$.
\end{corollary}

\begin{proof}
The estimates on $\sup_{0\leq t\leq T}\|v_y\|_2^2$ and $\int_0^T\|\sqrt{\varrho_0}v_t\|_2^2dt$ follow from Propositions \ref{PropEstG} and \ref{PropJPiInfty} by noticing that $v_y=\frac J\mu(G+\pi)$ and $\sqrt{\varrho_0}v_t=\frac{G_y}{\sqrt{\varrho_0}}$.
As for the estimate of $v_{yy}$, noticing that
$$
v_{yy}=\left(J\frac{v_y}{J}\right)_y=J\left(\frac{v_y}{J}\right)_y+\frac{v_y}{J}J_y=\frac J\mu(G_y+\pi_y)+\frac{J_y}{\mu}(G+\pi),
$$
it follows from Propositions \ref{PropEstG}--\ref{PropJPiy} that
\begin{equation*}
  \int_0^T\|v_{yy}\|_2^2dt\leq C\int_0^T[\|G_y\|_2^2+\|\pi_y\|_2^2+\|J_y\|_2^2(\|G\|_\infty^2+\|\pi\|_\infty^2)]dt\leq C.
\end{equation*}
The estimate for $J_t$ follows directly from (\ref{EQJ}) and the estimates obtained. By Propositions \ref{PropEstG}--\ref{PropJPiy}, it follows from (\ref{EQpi-v1}) that
\begin{eqnarray*}
  &&\int_0^T\|\pi_t\|_\infty^4dt\leq C\int_0^T(\|G\|_\infty^4+\|\pi\|_\infty^4)dt\leq C,
\end{eqnarray*}
and
\begin{eqnarray*}
  &&\int_0^T\|\pi_{yt}\|_2^{\frac43}dt\leq C\int_0^T\big(\|\pi\|_\infty\|\pi_y\|_2+\|G\|_\infty\|G_y\|_2\big)^\frac43dt \\
  &\leq& C\left(\int_0^T(\|\pi\|_\infty^4+\|G\|_\infty^4)dt\right)^{\frac13}\left(\int_0^T(\|\pi_y\|_2^2+\|G_y\|_2^2)dt\right)^{\frac23}
  \leq C.
\end{eqnarray*}
This completes the proof.
\end{proof}

The following $t$-weighted estimates will be used in the compactness arguments in the passage of taking limit from the non-vacuum case to the vacuum case.

\begin{proposition}
  \label{PropWei}
  The following estimate holds
  $$
  \int_0^Tt\|v_{yt}\|_2^2dt\leq C,
  $$
  for a positive constant $C$ depending only on $\gamma, \mu, \bar\varrho, \ell_0, \underline J, \bar J, m_0, \|\pi_0\|_\infty, E_0$, $\|G_0\|_2$, $\|J_0'\|_2$, $\|\pi_0'\|_2$, and $T$.
\end{proposition}

\begin{proof}
  Multiplying (\ref{EQG}) with $JG_t$, then integrating by parts yields
  \begin{eqnarray*}
    \frac\mu2\frac{d}{dt}\left\|\frac{G_y}{\sqrt{\varrho_0}}\right\|_2^2+\|\sqrt JG_t\|_2^2=-\gamma\int_0^Lv_yGG_tdy\\
    \leq\frac12\|\sqrt JG_t\|_2^2+C\|G\|_\infty^2\|v_y\|_2^2,
  \end{eqnarray*}
  which, multiplied with $t$, gives
  \begin{eqnarray*}
    \mu\frac{d}{dt}\left(t\left\|\frac{G_y}{\sqrt{\varrho_0}}\right\|_2^2\right)+t\|\sqrt JG_t\|_2^2\leq\mu\left\|\frac{G_y}{\sqrt{\varrho_0}}\right\|_2^2+Ct\|G\|_\infty^2\|v_y\|_2^2.
  \end{eqnarray*}
  Integrating the above with respect to $t$, and using Proposition \ref{PropEstG} and Corollary \ref{CorVy} yield
  \begin{equation}
  \label{N5}
    \sup_{0\leq t\leq T}\left(t\left\|\frac{G_y}{\sqrt{\varrho_0}}\right\|_2^2\right)+\int_0^Tt\|\sqrt JG_t\|_2^2dt\leq C.
  \end{equation}
  Recalling the expression of $G$, by direct calculations, and using (\ref{EQpi}), one deduces
  \begin{eqnarray*}
    G_t&=&\mu\left(\frac{v_{yt}}{J}-\frac{J_t}{J^2}v_y\right)-\pi_t\\
   &=&\mu\frac{v_{yt}}{J}-\mu\left(\frac{v_y}{J}\right)^2-\left(\mu(\gamma-1)\left|\frac{v_y}{J}\right|^2-\gamma\frac{v_y}{J}\pi
   \right)\\
   &=&\mu\frac{v_{yt}}{J}-\gamma\frac{v_y}{J}G,
  \end{eqnarray*}
  which gives
  $$
  v_{yt}=\frac1\mu(JG_t+\gamma v_yG).
  $$
  Therefore, it follows from (\ref{N5}), Proposition \ref{PropEstG}, and Corollary \ref{CorVy} that
  \begin{eqnarray*}
    \int_0^Tt\|v_{yt}\|_2^2dt\leq C\int_0^T(t\|\sqrt JG_t\|_2^2+t\|v_y\|_2^2\|G\|_\infty^2)dt\leq C,
  \end{eqnarray*}
  proving the conclusion.
\end{proof}

In summary, we have the following
\begin{corollary}
  \label{CorApri}
The following estimates hold
\begin{eqnarray*}
  &&\inf_{(y,t)\in(0,L)\times(0,T)}J\geq Ce^{-CT},\\
  &&\sup_{0\leq t\leq T}(\|J\|_{H^1}^2+\|J_t\|_2^2)+\int_0^T\|J_{t}\|_{H^1}^2dt\leq C,\\
  &&\sup_{0\leq t\leq T}\|v\|_{H^1}^2+\int_0^T(\|\sqrt{\varrho_0}v_t\|_2^2+\|v\|_{H^2}^2+t\|v_t\|_{H^1}^2)dt\leq C,\\
  &&\sup_{0\leq t\leq T}\|\pi\|_{H^1}^2+\int_0^T(\|\pi_t\|_\infty^4+\|\pi_{t}\|_{H^1}^{\frac43})dt\leq C,
\end{eqnarray*}
for a positive constant $C$ depending only on $\gamma, \mu, \bar\varrho, \underline J, \|(J_0, v_0, \pi_0)\|_{H^1},$ and $T$.
\end{corollary}

\begin{proof}
  This is a direct corollary of Propositions \ref{PROPEstJ}, \ref{PropJPiInfty}, \ref{PropJPiy}, \ref{PropWei}, and Corollary \ref{CorVy}, by using some necessary embedding inequalities.
\end{proof}

\begin{remark}
  \label{Remark}
Checking the proofs of Propositions \ref{PropEstG}--\ref{PropWei}, one can easily see that all the
constants $C$ in the arguments viewing as functions of $T$ can be chosen in such a way that are
continuous in $T\in[0,\infty)$.
\end{remark}

We conclude this section with the following global well-posedness result for the non-vacuum case.

\begin{theorem}
\label{THMNONVACUUM}
Under the conditions in Proposition \ref{PropLocal}, there is a unique global solution
$(J,v,\pi)$ to system (\ref{EQJ})--(\ref{EQpi}), subject to (\ref{BC})--(\ref{IC}), satisfying
 \begin{eqnarray*}
&0<J\in C([0,\infty); H^1),\quad  J_t\in L_{\text{loc}}^\infty([0,\infty); L^2)\cap L^2_{\text{loc}}([0,\infty); H^1),\\
&v\in C([0,\infty); H^1_0)\cap L_{\text{loc}}^2([0,\infty); H^2), \quad   v_t \in L_{\text{loc}}^2([0,\infty); L^2),\\
&\sqrt t  v_t \in L_{\text{loc}}^2([0,\infty); H^1),\\
&0\leq\pi\in C([0,\infty); H^1), \quad \pi_t\in L_{\text{loc}}^4([0,\infty); L^\infty)\cap L^{\frac43}_{\text{loc}}([0,\infty);H^1).
  \end{eqnarray*}
\end{theorem}

\begin{proof}
By Proposition \ref{PropLocal}, there is a unique local solution $(J,v,\pi)$ to system (\ref{EQJ})--(\ref{EQpi}), subject to (\ref{BC})--(\ref{IC}). By iteratively applying Proposition \ref{PropLocal}, one can extend the local solution to the maximal time of existence $T_\text{max}$. We claim that $T_{\text{max}}=\infty$. Assume by contradiction that $T_\text{max}<\infty$.
Then, by Corollary \ref{CorApri} and recalling Remark \ref{Remark}, there is a positive constant $C$, independent of
$T\in(0,T_\text{max})$, such that
\begin{eqnarray*}
  &&\inf_{(y,t)\in(0,L)\times(0,T)}J\geq Ce^{-CT},\\
  &&\sup_{0\leq t\leq T}(\|J\|_{H^1}^2+\|J_t\|_2^2)+\int_0^T\|J_{t}\|_{H^1}^2dt\leq C,\\
  &&\sup_{0\leq t\leq T}\|v\|_{H^1}^2+\int_0^T(\|\sqrt{\varrho_0}v_t\|_2^2+\|v\|_{H^2}^2+t\|v_t\|_{H^1}^2)dt\leq C,\\
  &&\sup_{0\leq t\leq T}\|\pi\|_{H^1}^2+\int_0^T(\|\pi_t\|_\infty^4+\|\pi_{t}\|_{H^1}^{\frac43})dt\leq C,
\end{eqnarray*}
Thanks to this, by the local existence result, Proposition \ref{PropLocal}, one can extend the local solution $(J, v, \pi)$ beyond $T_\text{max}$, contradicting to the definition of $T_\text{max}$. Therefore, it must have $T_\text{max}=\infty$. This proves the conclusion.
\end{proof}

\section{Global well-posedness: in the presence of vacuum}

In this section, we prove our main result as follows.

\begin{proof}[\textbf{Proof of Theorem \ref{thm}}]
\textbf{Existence. }Choose $\varrho_{0n}\in H^1$, with $\frac1n\leq\varrho_{0n}\leq\bar\varrho+1$, such that $\varrho_{0n}\rightarrow\varrho_0$ in $L^q$, for any $q\in(1,\infty)$. By Theorem \ref{THMNONVACUUM}, for any $n$, there is a unique global solution $(J_n, v_n, \pi_n)$ to system (\ref{EQJ})--(\ref{EQpi}), subject to (\ref{BC})--(\ref{IC}), with $\varrho_0$ in (\ref{EQv}) replaced with
$\varrho_{0n}$. By Corollary \ref{CorApri}, there is a positive constant $C$, independent of $n$, such that
\begin{eqnarray}
  &&\inf_{(y,t)\in(0,L)\times(0,T)}J_n\geq Ce^{-CT},\nonumber\\
  &&\sup_{0\leq t\leq T}(\|J_n\|_{H^1}^2+\|\partial_tJ_n\|_2^2)+\int_0^T\|\partial_tJ_n\|_{H^1}^2dt\leq C,\nonumber\\
  &&\sup_{0\leq t\leq T}\|v_n\|_{H^1}^2+\int_0^T(\|\sqrt{\varrho_0}\partial_tv_n\|_2^2+\|v_n\|_{H^2}^2+t\|\partial_tv_n\|_{H^1}^2)dt\leq C,\label{AP3}\\
  &&\sup_{0\leq t\leq T}\|\pi_n\|_{H^1}^2+\int_0^T\|\partial_{t}\pi_n\|_{H^1}^{\frac43}dt\leq C,\nonumber
\end{eqnarray}
for any $T\in(0,\infty)$. By the Aubin-Lions lemma, and using Cantor's diagonal argument, there is a subsequence, still denoted
by $(J_n, v_n, \pi_n)$, and $(J, v, \pi)$ enjoying the regularities
\begin{align}
  &J\in L^\infty(0,T; H^1),\quad J_t\in L^\infty(0,T; L^2)\cap L^2(0,T; H^1), \label{REG1}\\
  &v\in L^\infty(0,T; H^1)\cap L^2(0,T; H^2), \quad\sqrt tv_t\in L^2(0,T; H^1), \label{REG1-2}\\
  &\pi\in L^\infty(0,T; H^1), \quad \pi_t\in L^{\frac43}(0,T; H^1), \label{REG2}
\end{align}
such that
\begin{eqnarray}
  J_n\overset{*}{\rightharpoonup}J, \quad\mbox{in }L^\infty(0,T; H^1), \quad\partial_tJ_n\overset{*}{\rightharpoonup}J_t, \quad\mbox{in }L^\infty(0,T; L^2),\label{WCG1}\\
  \partial_tJ_n\rightharpoonup J_t, \quad\mbox{in }L^2(0,T;H^1),\label{WCG2}\\
  v_n\overset{*}{\rightharpoonup}v,\quad\mbox{in }L^\infty(0,T; H^1), \quad v_n\rightharpoonup v\quad\mbox{in }L^2(0,T;H^2), \label{WCG3}\\
  \partial_tv_n\rightharpoonup v_t,\quad\mbox{in }L^2(\delta, T; H^1),\quad\forall\delta\in(0,T),\label{WCG4}\\
  \pi_n\overset{*}{\rightharpoonup}\pi, \quad\mbox{in }L^\infty(0,T;H^1),\quad\partial_t\pi_n\rightharpoonup \pi_t,\quad\mbox{in }L^{\frac43}(0,T;H^1),\label{WCG5}
\end{eqnarray}
and
\begin{eqnarray}
  &J_n\rightarrow J,\quad\mbox{in }C([0,T]; C([0,L])),\label{SC1}\\
  &v_n\rightarrow v,\quad\mbox{in }C([\delta,T]; C([0,L]))\cap L^2(\delta,T; H^1),\quad\forall\delta\in(0,T),\label{SC2} \\
  &\pi_n\rightarrow\pi,\quad\mbox{in }C([0,T]; C([0,T])). \label{SC3}
\end{eqnarray}
Here, $\rightarrow$, $\rightharpoonup$, and $\overset{*}{\rightharpoonup}$ denote, respectively, the strong, weak, and weak*
convergence in the corresponding spaces.
Thanks to (\ref{WCG1})--(\ref{SC2}), one can take the limit $n\rightarrow\infty$ to show
that $(J, v, \pi)$ is a solution to system (\ref{EQJ})--(\ref{EQpi}), on $(0,L)\times(0,T)$. Moreover, recalling $(J_n,
\pi_n)|_{t=0}=(J_0, \pi_0)$, it is clear from (\ref{SC1}) and (\ref{SC3}) that $(J,
\pi)|_{t=0}=(J_0, \pi_0)$.

One needs to verify the regularities of $(J,v,\pi)$ and that $(\varrho_0v)|_{t=0}=\varrho_0v_0$.
Using (\ref{AP3}) and (\ref{WCG4}), by the lower semi-continuity of the norms, one deduces
\begin{eqnarray*}
  \int_\delta^T\|\sqrt{\varrho_0}v_t\|_2^2dt&\leq&\varliminf_{n\rightarrow\infty}\int_\delta^T\|\sqrt{\varrho_{0n}}\partial_t v_n\|_2^2dt\leq C,
\end{eqnarray*}
for any $\delta\in(0,T)$, and for a positive constant $C$ independent of $\delta$, and, thus,
$\sqrt{\varrho_0}v_t\in L^2(0,T;L^2).$
The desired regularities $J,\pi\in C([0,T];H^1)$ follow from  (\ref{REG1}) and (\ref{REG2}).

It remains to verify $\varrho_0v\in C([0,T];L^2)$ and $(\varrho_0v)|_{t=0}=\varrho_0v_0.$ To this end, noticing that
(\ref{REG1-2}) and (\ref{REG2}) imply $v\in C((0,T]; H^1)$, it suffices to show that $(\varrho_0v)(\cdot,t)\rightarrow\varrho_0v_0$, strongly
in $L^2$, as $t\rightarrow0$.
Using (\ref{AP3}), it follows
\begin{eqnarray}
  \|\varrho_{0n}(v_n-v_0)\|_2&=&\left\|\varrho_{0n}\int_0^t\partial_tv_nds\right\|_2\leq C\int_0^t\|
  \sqrt{\varrho_{0n}}\partial_tv_n\|_2ds\nonumber\\
  &\leq& C\sqrt t\|\sqrt{\varrho_{0n}}\partial_tv_n\|_{L^2(0,T;L^2)}\leq C\sqrt t, \label{IVV1}
\end{eqnarray}
for a positive constant $C$ independent of $n$. Recalling (\ref{SC2}) and $\varrho_{0n}\rightarrow\varrho_0$, for any $q>1$,
one has
\begin{equation}
\label{IVV2}
(\varrho_{0n}v_n)(\cdot,t)\rightarrow(\varrho_0v)(\cdot,t),
\quad\mbox{in}\quad L^2, \quad\forall t>0.
\end{equation}
It follows from (\ref{IVV1}) that
\begin{eqnarray*}
  \|\varrho_0(v-v_0)\|_2(t)&\leq&\|\varrho_0v-\varrho_{0n}v_n\|(t)+\|\varrho_{0n}(v_n-v_0)\|_2(t)
  +\|(\varrho_{0n}-\varrho_0)v_0\|_2\\
  &\leq&\|\varrho_0v-\varrho_{0n}v_n\|(t)+C\sqrt t
  +C\|\varrho_{0n}-\varrho_0\|_2,
\end{eqnarray*}
where $C$ is independent of $n$, from which, recalling (\ref{IVV2}), one can take the limit $n\rightarrow\infty$ to get
$$
\|\varrho_0(v-v_0)\|_2(t)\leq C\sqrt t.
$$
This proves the continuity of $\varrho_0v$ at $t=0$ and verifies $\varrho_0v|_{t=0}=\varrho_0v_0.$

Therefore, $(J, v, \pi)$ is a global
solution to system (\ref{EQJ})--(\ref{EQpi}), subject to the initial and boundary conditions (\ref{BC})--(\ref{IC}),
satisfying the regularities stated in Theorem \ref{thm}. This proves the existence part of Theorem \ref{thm}.

\textbf{Uniqueness.} Let $(J_1, v_1, \pi_1)$ and $(J_2, v_2, \pi_2)$ be two solutions to system (\ref{EQJ})--(\ref{EQpi}),
subject to (\ref{BC})--(\ref{IC}), and denote $(J, v, \pi):=(J_1-J_2, v_1-v_2, \pi_1-\pi_2)$. Then, straightforward calculations lead to
\begin{align}
  &J_t=v_y, \label{D1}\\
  &\varrho_0v_t-\mu\left(\frac{v_y}{J_1}\right)_y+\mu\left(\frac{J v_{2y}}{J_1J_2}\right)_y+\pi_y=0,\label{D2}\\
  &\pi_t+\gamma\left(\frac{\pi v_{1y}}{J_1}+\frac{\pi_2 v_y}{J_1}-\frac{J\pi_2v_{2y}}{J_1J_2}\right)
  =\mu(\gamma-1)\left(\frac{v_{1y}}{J_1}+\frac{v_{2y}}{J_2}\right)\left(\frac{v_y}{J_1}-\frac{J v_{2y}}{J_1J_2}\right).
  \label{D3}
\end{align}
Multiplying (\ref{D1}), (\ref{D2}), and (\ref{D3}), respectively, with $J, v,$ and $\pi$, and integrating the resultants
over $(0,L)$, one gets from integration by parts and using the Young inequalities that
\begin{eqnarray*}
  &&\frac12\frac d{dt}\|J\|_2^2\leq\varepsilon\|v_y\|_2^2+C_\varepsilon\|J\|_2^2, \\
  &&\frac12\frac d{dt}\|\sqrt{\varrho_0}v\|_2^2+\mu\left\|\frac{v_y}{\sqrt{J_1}}\right\|_2^2
  \leq\varepsilon\|v_y\|_2^2+C_\varepsilon(\|\pi\|_2^2+\|v_{2y}\|_\infty^2\|J\|_2^2),\\
  &&\frac12\frac d{dt}\|\pi\|_2^2\leq\varepsilon\|v_y\|_2^2+C_\varepsilon(\|v_{1y}\|_\infty^2+\|v_{2y}\|_\infty^2+\|\pi_2\|_\infty^2)
  (\|J\|_2^2+\|\pi\|_2^2),
\end{eqnarray*}
where the fact that $J_1$ and $J_2$ have positive lower bounds on $(0,L)\times(0,T)$ for any finite $T$ has been used.
Adding up the previous three inequalities and choosing $\varepsilon$ sufficiently small, one obtains
\begin{eqnarray*}
  &&\frac d{dt}(\|J\|_2^2+\|\sqrt{\varrho_0}v\|_2^2+\|\pi\|_2^2)+\mu\left\|\frac{v_y}{\sqrt{J_1}}\right\|_2^2\\
  &\leq& C(1+\|v_{1y}\|_\infty^2+\|v_{2y}\|_\infty^2+\|\pi_2\|_\infty^2)
  (\|J\|_2^2+\|\pi\|_2^2),
\end{eqnarray*}
from which, noticing that $\pi_i, v_{iy}\in L^2(0,T;L^\infty), i=1, 2$, and by the Gronwall inequality,
one obtains $J\equiv\pi\equiv\sqrt{\varrho_0}v\equiv v_y\equiv0$. Thanks to this, by the Poincar\'e inequality, the
uniqueness follows.
\end{proof}

\section{Appendix: local well-posedness, i.e., proof of Proposition \ref{PropLocal}}

In this appendix, we prove the local well-posedness of system (\ref{EQJ})--(\ref{EQpi}), subject to (\ref{BC})--(\ref{IC}), for the case that the the initial density $\varrho_0$ is uniformly away from vacuum. In other words, we give the proof of Proposition \ref{PropLocal}.

For positive time $T\in(0,\infty)$, denote
$$
  Q_T:=(0,L)\times(0,T),\quad X_T:=L^\infty(0,T; H_0^1)\cap L^2(0,T; H^2),
$$
and
$$
\|f\|_{V_T}:=\left(\sup_{0\leq t\leq T}\|f\|_2^2+\int_0^T\|f_{y}\|_2^2dt\right)^{\frac12}.
$$
For positive numbers $M$ and $T$, we denote
$$
\mathscr K_{M,T}:=\left\{v\in X_T, \|v_y\|_{V_T}\leq M\right\}.
$$
By the Poincar\'e inequality, one can verify that $\mathscr K_{M,T}$ is a closed subset of $X_T$.

Given $(\varrho_0, J_0, v_0, \pi_0)$, satisfying
\begin{eqnarray}
  &0<\underline\varrho\leq\varrho_0\leq\bar\varrho<\infty,\quad 0<\underline J\leq J_0\leq\bar J<\infty,
  \label{H1}\\
   &\pi_0\geq0,\quad (\varrho_0, J_0, \pi_0)\in H^1,\quad v_0\in H_0^1,\label{H2}
\end{eqnarray}
for positive numbers $\underline\varrho, \bar\varrho, \underline J,$ and $\bar J$.

Define three mappings $\mathscr Q, \mathscr R,$ and $\mathscr F$ as follows. First, for $v\in\mathscr K_{M,T}$,
define $J=\mathscr Q(v)$ as the unique solution to
$$
J_t=v_y,\quad J|_{t=0}=J_0.
$$
Next, for given $v\in\mathscr K_{M,T}$, and with $J$ solved as above, define $\pi=\mathscr R(v)$ as the unique solution to
\begin{eqnarray*}
  \pi_t+\gamma\frac{v_y}{J}\pi=\mu(\gamma-1)\left(\frac{v_y}{J}\right)^2,\quad
  \pi|_{t=0}=\pi_0.
\end{eqnarray*}
And finally, for given $v\in\mathscr K_{M,T}$, and with $J$ and $\pi$ solved as above, define $V=\mathscr F(v)$ as the
unique solution to
\begin{equation}\label{L6}
\left\{
  \begin{array}{ll}
    V_t-\mu\frac{ V_{yy}}{J\varrho_0}=-\left(\mu\frac{J_y v_y}{J^2\varrho_0 }+\frac{\pi_y}{\varrho_0}\right),&\mbox{in }Q_T,\\
    V(0,t)=V(L,t)=0,&t\in(0,T),\\
    V(y,0)=v_0(y),&y\in(0,L).
  \end{array}
\right.
\end{equation}

It is clear that
\begin{eqnarray*}
 \mathscr Q(v)=J_0+\int_0^tv_yds,\quad
\mathscr R(v)=\mathscr R_1(v)+\mu(\gamma-1)\mathscr R_2(v), \label{R}
\end{eqnarray*}
where
$$
  \left\{
  \begin{array}{l}\mathscr R_1(v)=\pi_0\exp\left\{-\gamma\int_0^t\frac{v_y}{J}ds\right\},
  \\
  \mathscr R_2(v)=\int_0^t\left(\frac{v_y}{J}
  \right)^2\exp\left\{-\gamma\int_\tau^t\frac{v_y}{J}ds\right\}d\tau,
  \end{array}
  \right.\quad\mbox{with }J=\mathscr Q(v).
$$
%It is clear that $J=\mathscr Q(v)$ and $\pi=\mathscr R(v)$ solve
%
%
%Define another mapping $\mathscr F$, with $V=\mathscr F(v)$ being the unique solution to the following initial boundary value problem:

In order to prove the local existence and uniqueness
of solutions to system (\ref{EQJ})--(\ref{EQpi}), subject to (\ref{BC})--(\ref{IC}), and recalling the definitions of the mappings $\mathscr Q, \mathscr R$, and $\mathscr F$, it suffices to show that the mapping $\mathscr F$ has a unique fixed point in $X_T$, which will be proved by the contractive mapping principle.

For simplicity of notations, throughout this section, we agree the following:
\begin{eqnarray*}
&J=\mathscr Q(v),\quad\pi=\mathscr R(v), \quad J_i=\mathscr Q(v_i),\quad\pi_i=\mathscr R(v_i), \quad i=1,2, \\
&\delta J=J_1-J_2,\quad \delta\pi=\pi_1-\pi_2,   \quad \delta v=v_1-v_2,
\end{eqnarray*}
for arbitrary $v, v_1, v_2\in\mathscr K_{M,T}$.
By the Poinc\'are and Gagliardo-Nirenberg inequality, there is a positive constant $C_1$ depending
only on $L$, such that
\begin{equation}
\|v_{y}\|_\infty\leq C_1\|v_{y}\|_2^{\frac12}\|v_{yy}\|_2^{\frac12}.\label{EBBINQ}
\end{equation}
This kind inequality for $v$ will be frequently used without further mentions, and we use $C_1$
specifically to denote the constant in the above inequality.

In the rest of this section, we always assume that $M$ and $T$ are two positive constants, to be determined later, satisfying
\begin{equation}
  MT^{\frac14}\leq 1, \quad T\leq 1. \label{H}
\end{equation}

\begin{proposition}
  \label{PropA0}
(i) For any $v\in X_T$, it follows that
$$
\| v_y \|_{L^2(0,T; L^\infty)}\leq C_1T^{\frac14}\| v_y \|_{V_T},\quad \| v_y \|_{L^1(0,T; L^\infty)}\leq C_1T^{\frac34}\| v_y \|_{V_T}.
$$
(ii) Consequently, for any $v\in\mathscr K_{M,T}$, one has
\begin{eqnarray*}
   \|v_y\|_{L^2(0,T;L^\infty)}\leq C_1,\quad
   \|v_y\|_{L^1(0,T;L^\infty)}\leq C_1.
\end{eqnarray*}

\end{proposition}

\begin{proof}
  For any $v\in X_T$, by the H\"older inequality and (\ref{EBBINQ}), one deduces
  \begin{eqnarray*}
    \|v_y\|_{L^2(0,T;L^\infty)}&=&\left(\int_0^T\|v_y\|_\infty^2dt\right)^{\frac12}
    \leq C_1\left(\int_0^T\|v_y\|_2\|v_{yy}\|_2dt\right)^{\frac12}\nonumber\\
    &\leq&C_1\left[\sup_{0\leq t\leq T}\|v_y\|_2^2\left(\int_0^T\|v_{yy}\|_2^2dt\right)^{\frac12}T^{\frac12}\right]^{\frac12}
    \leq C_1T^{\frac14}\| v_y \|_{V_T},
  \end{eqnarray*}
  which leads to the first inequality in (i).
  The second inequality in (i) follows from the first one by simply applying the H\"older inequality. The inequalities in (ii)
  follow from those in (i) by using the conditions in (\ref{H}).
\end{proof}

\subsection{Properties of $\mathscr Q$}
\begin{proposition}
  \label{PropA1}
(i) It holds that
\begin{eqnarray*}
  \|\partial_y\mathscr Q(v)\|_{L^\infty(0,T; L^2)}&\leq&\|J_0'\|_2+1, \\
  \|\mathscr Q(v_1)-\mathscr Q(v_2)\|_{L^\infty(Q_T)}&\leq& C_1T^{\frac34}\|\partial_y(v_1-v_2)\|_{V_T},\\
  \|\partial_y(\mathscr Q(v_1)-\mathscr Q(v_2))\|_{L^\infty(0,T; L^2)}&\leq& T^{\frac12}\|\partial_y(v_1-v_2)\|_{V_T},
\end{eqnarray*}
for any $v, v_1, v_2\in\mathscr K_{M,T}$.

(ii) Assume, in addition, that $T\leq\left(\frac{\underline J}{2C_1}\right)^2$. Then,
  $$
  \frac{\underline J}{2}\leq \mathscr Q(v)\leq 2\bar J,  \quad\mbox{on }Q_T,
  $$
  for any $v\in\mathscr K_{M,T}$.
\end{proposition}

\begin{proof}
  (i) Recalling the expression of $\mathscr Q$, it is clear that
  \begin{eqnarray*}
    \|\partial_y\mathscr Q(v)\|_2&=&\left\|J_0'+\int_0^tv_{yy}d\tau\right\|_2\leq\|J_0'\|_2+\int_0^t\|v_{yy}\|_2d\tau\\
    &\leq&\|J_0'\|_2+T^{\frac12}\|v_{yy}\|_{L^2(Q_T)}\leq \|J_0'\|_2+T^{\frac12}M\leq\|J_0'\|_2+1,
  \end{eqnarray*}
  where (\ref{H}) has been used. Similarly,
  \begin{eqnarray*}
  \|\partial_y(\mathscr Q(v_1)-\mathscr Q(v_2))\|_2&=&\left\|\int_0^t(v_1-v_2)_{yy}d\tau\right\|_2
  \leq T^{\frac12}\|(v_1-v_2)_{yy}\|_{L^2(Q_T)}\\
  &\leq& T^{\frac12}\|(v_1-v_2)_y\|_{V_T}.
  \end{eqnarray*}
  By (i) of Proposition \ref{PropA0}, one deduces
  \begin{eqnarray*}
    \|\mathscr Q(v_1)-\mathscr Q(v_2)\|_\infty&=&\left\|\int_0^t(v_1-v_2)_{y}d\tau\right\|_\infty \leq\|(v_1-v_2)_{y}\|_{L^1(0,T;L^\infty)}\\
    &\leq& C_1T^{\frac34}\|(v_1-v_2)_y\|_{V_T}.
  \end{eqnarray*}

  (ii) If $T\leq (\frac{\underline J}{2C_1} )^2$, it follows from (ii) of Proposition \ref{PropA0} that
$$
    \left\|\int_0^tv_yd\tau\right\|_\infty \leq
     T^{\frac12}\|v_y\|_{L^2(0,T;L^\infty)} \leq C_1T^{\frac12}\leq\frac{\underline J}{2},
$$
and, consequently,
  \begin{eqnarray*}
    &\mathscr Q(v)=J_0+\int_0^tv_yd\tau\geq\underline J-\frac{\underline J}{2}=\frac{\underline J}{2},\\
    &\mathscr Q(v)=J_0+\int_0^tv_yd\tau\leq\overline J+\frac{\underline J}{2}\leq 2\bar J,
  \end{eqnarray*}
proving the conclusion.
\end{proof}

Due to Proposition \ref{PropA1}, in the rest of this section, we always assume, in addition to (\ref{H}), that
$T\leq\left(\frac{\underline J}{2c_1}\right)^2$, so that (ii) of Proposition \ref{PropA1} applies.

\begin{proposition}
  \label{PropA2}
  The following estimates hold:
$$
    \left\|\frac{v_{1y}}{\mathscr Q(v_1)}-\frac{v_{2y}}{\mathscr Q(v_2)}\right\|_{L^2(0,T; L^\infty)}\leq CT^{\frac14}\|\partial_y(v_1-v_2)\|_{V_T},
$$
and
\begin{eqnarray*}
\left\|\exp\left\{-\gamma\int_\tau^t\frac{v_{1y}}{\mathscr Q(v_1)}ds\right\}-\exp\left\{-\gamma\int_\tau^t\frac{v_{2y}}{\mathscr Q(v_2)}ds\right\}\right\|_\infty\\
\leq CT^{\frac34}\|\partial_y(v_1-v_2)\|_{V_T},
\end{eqnarray*}
  for any $0\leq\tau<t\leq T$, and $v_1, v_2\in\mathscr K_{M,T}$, where $C$ is a positive constant depending only on $\gamma, L, \underline J$.
\end{proposition}

\begin{proof}
Applying Proposition \ref{PropA0} and Proposition \ref{PropA1},
one deduces
\begin{eqnarray}
  &&\left\|\frac{v_{1y}}{\mathscr Q(v_1)}-\frac{v_{2y}}{\mathscr Q(v_2)}\right\|_{L^2(0,T; L^\infty)}\nonumber\\
  &=&\left\|\frac{(v_1-v_2)_y}{\mathscr Q(v_1)}-\frac{(\mathscr Q(v_1)-\mathscr Q(v_2))v_{2y}}{\mathscr Q(v_1)\mathscr Q(v_2)}\right\|_{L^2(0,T;L^\infty)}\nonumber\\
  &\leq&\frac{2}{\underline J}\|(v_1-v_2)_y\|_{L^2(0,T; L^\infty)}+\left(\frac{2}{\underline J}\right)^2\|\mathscr Q(v_1)-\mathscr Q(v_2)\|_{L^\infty(
  Q_T)}\|v_{2y}\|_{L^2(0,T;L^\infty)}\nonumber\\
  &\leq&\left(\frac{2C_1}{\underline J}T^{\frac14}+\left(\frac{2C_1}{\underline J}\right)^2T^{\frac34}\right)\|(v_1-v_2)_y\|_{V_T}\nonumber\\
  &\leq& CT^{\frac14}\|(v_1-v_2)_y\|_{V_T}. \label{L2}
\end{eqnarray}
By the mean value theorem, there is a number $\eta\in(0,1)$, such that
\begin{eqnarray*}
&&\exp\left\{-\gamma\int_\tau^t\frac{v_{1y}}{\mathscr Q(v_1)}ds\right\}-\exp\left\{-\gamma\int_\tau^t\frac{v_{2y}}{\mathscr Q(v_2)}ds\right\}\\
&=&-\gamma \exp\left\{-\gamma\int_\tau^t\left(\eta\frac{v_{1y}}{\mathscr Q(v_1)}+(1-\eta)\frac{v_{2y}}{\mathscr Q(v_2)}\right)ds\right\}
\int_\tau^t\left(\frac{v_{1y}}{\mathscr Q(v_1)}-\frac{v_{2y}}{\mathscr Q(v_2)}\right)ds.
\end{eqnarray*}
Thus, using (\ref{L2}), it follows from Proposition \ref{PropA0} and Proposition \ref{PropA1} that
\begin{eqnarray*}
 &&\left\|\exp\left\{-\gamma\int_\tau^t\frac{v_{1y}}{\mathscr
  Q(v_1)}ds\right\}-\exp\left\{-\gamma\int_\tau^t\frac{v_{2y}}{\mathscr Q(v_2)}ds\right\}\right\|_\infty\\
  &\leq&\gamma e^{\gamma\left(\eta\left\|\frac{v_{1y}}{\mathscr Q(v_1)}\right\|_{L^1(0,T; L^\infty)}+(1-\eta)
  \left\|\frac{v_{2y}}{\mathscr Q(v_2)}\right\|_{L^1(0,T; L^\infty)}\right)}\left\|\frac{v_{1y}}{\mathscr Q(v_1)}-\frac{v_{2y}}{\mathscr Q(v_2)}
  \right\|_{L^1(0,T; L^\infty)}\\
  &\leq&\gamma e^{\frac{2\gamma C_2}{\underline J}}T^{\frac12}\left\|\frac{v_{1y}}{\mathscr Q(v_1)}-\frac{v_{2y}}{\mathscr Q(v_2)}
  \right\|_{L^2(0,T; L^\infty)}\leq CT^{\frac34}\|(v_1-v_2)_y\|_{V_T},
\end{eqnarray*}
proving the conclusion.
\end{proof}

\begin{proposition}
The following estimates hold
  \label{PropA3}
\begin{eqnarray*}
\Big\|\Big(\frac{v_{y}}{\mathscr Q(v)}\Big)_y\Big\|_{L^2(Q_T)\cap L^1(0,T;L^2)} \leq  C(1+M+\|J_0'\|_2), \\
\Big\|\Big(\frac{v_{1y}}{\mathscr Q(v_1)}-\frac{v_{2y}}{\mathscr Q(v_2)}\Big)_y\Big\|_{L^2(Q_T)} \leq  C(1+\|J_0'\|_2)\|\partial_y(v_1-v_2)\|_{V_T},
\end{eqnarray*}
and
\begin{eqnarray*}
 \left\|\partial_y\left(\exp\left\{-\gamma\int_\tau^t\frac{v_{1y}}{\mathscr Q(v_1)}ds\right\}-\exp\left\{-\gamma\int_\tau^t\frac{v_{2y}}{\mathscr Q(v_2)}ds\right\}\right)\right\|_{L^\infty(0,T;L^2)}\\
\leq C(1+\|J_0'\|_2)T^{\frac12}\|\partial_y(v_1-v_2)\|_{V_T},
\end{eqnarray*}
for any $0\leq\tau<t\leq T$, and for any $v, v_1, v_2\in\mathscr K_{M,T}$, where $C$ is a positive constant depending only on $\gamma, L,$ and $\underline J$.
\end{proposition}

\begin{proof}
Note that $(\frac{v_{y}}{J})_y=\frac{v_{yy}}{J}-\frac{J_yv_{y}}{J^2}$, it follows from Proposition \ref{PropA0} and
Proposition \ref{PropA1} that
\begin{eqnarray*}
  \Big\|\Big(\frac{v_{y}}{\mathcal Q(v)}\Big)_y\Big\|_{L^2(Q_T)}&\leq&\Big\|\frac{v_{yy}}{\mathcal Q(v)}
\Big\|_{L^2(Q_T)}
  +\Big\|\frac{\partial_y\mathcal Q(v)}{\mathcal Q(v)^2}\Big\|_{L^\infty(0,T;L^2)}\|v_{y}\|_{L^2(0,T;L^\infty)}\\
  &\leq&\frac{2M}{\underline J}+\left(\frac{2}{\underline J}\right)^2(\|J_0'\|_2+1)C_1
  \leq C(1+M+\|J_0'\|_2).
\end{eqnarray*}
The $L^1(0,T; L^2)$ estimate for $(\frac{v_{y}}{\mathcal Q(v)})_y$ follows from the above inequality by simply
using the H\"older inequality.

For simplicity of notations, for $v_1, v_2\in\mathscr K_{M,T}$, we denote $\delta v=v_1-v_2$, $J_i=\mathscr Q(v_i)$, $i=1,2$, and $\delta J=J_1-J_2$.
By direct calculations
\begin{eqnarray*}
  &&\Big(\frac{v_{1y}}{J_1}-\frac{v_{2y}}{J_2}\Big)_y=\frac{v_{1yy}}{J_1}-\frac{J_{1y}v_{1y}}{J_1^2}-\left(\frac{v_{2yy}}{J_2}
  -\frac{J_{2y}v_{2y}}{J_2^2}\right)\\
  &=&\frac{\delta v_{yy}}{J_1}-\frac{\delta Jv_{2yy}}{J_1J_2}-\left(\frac{\delta J_yv_{1y}}{J_1^2}-J_{2y}v_{1y}\frac{(J_1+J_2)\delta J}{J_1^2J_2^2}+\frac{J_{2y}}{J_2^2}\delta v_y\right).
\end{eqnarray*}
Therefore, it follows from Proposition \ref{PropA0} and Proposition \ref{PropA1} that
\begin{eqnarray}
  &&\Big\|\Big(\frac{v_{1y}}{J_1}-\frac{v_{2y}}{J_2}\Big)_y\Big\|_{L^2(Q_T)}\nonumber\\
  &\leq&\frac{2}{\underline J}\|\delta v_{yy}\|_{L^2(Q_T)}+\left(\frac{2}{\underline J}\right)^2\|\delta J\|_{L^\infty(Q_T)}
  \|v_{2yy}\|_{L^2(Q_T)}\nonumber\\
  &&+\left(\frac{2}{\underline J}\right)^2\|\delta J_y\|_{L^\infty(0,T;L^2)}\|v_{1y}\|_{L^2(0,T;L^\infty)}\nonumber\\
  &&+4\bar J\left(\frac{2}{\underline J}\right)^4\|J_{2y}\|_{L^\infty(0,T;L^2)}
  \|\delta J\|_{L^\infty(Q_T)}\|v_{1y}\|_{L^2(0,T;L^\infty)}\nonumber\\
  &&+\left(\frac{2}{\underline J}\right)^2\|J_{2y}\|_{L^\infty(0,T;L^2)}\|\delta v_y\|_{L^2(0,T;L^\infty)}\nonumber\\
  &\leq& \frac{2}{\underline J}\|\delta v_y\|_{V_T}+\left(\frac{2}{\underline J}\right)^2MC_1T^{\frac34}\|\delta v_y\|_{V_T}
  \nonumber\\
  &&+\left(\frac{2}{\underline J}\right)^2C_1T^{\frac12}\|\delta v_y\|_{V_T}+4\bar J\left(\frac{2}{\underline J}\right)^4
  C_1^2(1+\|J_0'\|_2)T^{\frac34}\|\delta v_y\|_{V_T}\nonumber\\
  &&+\left(\frac{2}{\underline J}\right)^2
  (1+\|J_0'\|_2)C_1 T^{\frac14}\|\delta v_y\|_{V_T}\nonumber\\
  &\leq&C(1+\|J_0'\|_2)\|\delta v_y\|_{V_T}. \label{L3}
\end{eqnarray}
Straightforward computations yield
\begin{eqnarray*}
 &&\partial_y\left(\exp\left\{-\gamma\int_\tau^t\frac{v_{1y}}{J_1}ds\right\}-\exp\left\{-\gamma\int_\tau^t\frac{v_{2y}}{J_2}ds
\right\}\right)\\
 &=&-\gamma \left(\exp\left\{-\gamma\int_\tau^t\frac{v_{1y}}{J_1}ds\right\}-\exp\left\{-\gamma\int_\tau^t\frac{v_{2y}}{J_2}ds
\right\}\right)\int_\tau^t
 \left(\frac{v_{2y}}{J_2}\right)_yds \\
 &&-\gamma \exp\left\{-\gamma\int_\tau^t\frac{v_{1y}}{J_1}ds\right\}\int_\tau^t
 \left(\frac{v_{1y}}{J_1}-\frac{v_{2y}}{J_2}\right)_yds.
\end{eqnarray*}
Therefore, it follows from Propositions \ref{PropA0}, \ref{PropA2}, and \ref{PropA3} that
\begin{eqnarray*}
  &&\left\|\partial_y\left(\exp\left\{-\gamma\int_\tau^t\frac{v_{1y}}{J_1}ds\right\}-
\exp\left\{-\gamma\int_\tau^t\frac{v_{2y}}{J_2}ds\right\}\right)\right\|_2\\
  &\leq&\gamma\left\|\left(\frac{v_{2y}}{J_2}\right)_y\right\|_{L^1(0,T;L^2)}
\left\|\exp\left\{-\gamma\int_\tau^t\frac{v_{1y}}{J_1}ds\right\}-
\exp\left\{-\gamma\int_\tau^t\frac{v_{2y}}{J_2}ds\right\}\right\|_\infty\\
  &&+\gamma \exp\left\{\gamma\left\|\frac{v_{1y}}{J_1}\right\|_{L^1(0,T; L^\infty)}\right\}\left\|\left(\frac{v_{1y}}{J_1}-\frac{v_{2y}}{J_2}\right)_y\right\|_{L^1(0,T;L^2)}\\
  &\leq&\gamma e^{\frac{2\gamma C_1}{\underline J}}T^{\frac12}\left\|\left(\frac{v_{1y}}{J_1}-\frac{v_{2y}}{J_2}\right)_y\right\|_{L^2(Q_T)}+C(1+
M+\|J_0'\|_2)T^{\frac34}\|\delta v_y\|_{V_T}\\
  &\leq& C(1+\|J_0'\|_2)T^{\frac12}\|\delta v_y\|_{V_T},
\end{eqnarray*}
proving the conclusion.
\end{proof}

\subsection{Properties of $\mathscr R$}
\begin{proposition}
  \label{PropA4}
It holds that
$$
\|\partial_y(\mathscr R_1(v_1)-\mathscr R_1(v_2))\|_{L^2(Q_T)}\leq CT \|\partial_y(v_1-v_2)\|_{V_T},
$$
for any $v_1, v_2\in\mathscr K_{M,T}$, and for a positive constant $C$ depending only on $\gamma$, $L$, $\underline J$, $\|J_0'\|_2$, $\|\pi_0\|_\infty,$ and $\|\pi_0'\|_2$.
\end{proposition}

\begin{proof}
  For simplicity of notations, for $v_1, v_2\in\mathscr K_{M,T}$, we denote $\delta v=v_1-v_2$, $J_i=\mathscr Q(v_i)$, $i=1,2$, and $\delta J=J_1-J_2$. Note that
  \begin{eqnarray*}
  \partial_y(\mathscr R_1(v_1)-\mathscr R_1(v_2))&=&\pi_0\partial_y\left(\exp\left\{-\gamma\int_0^t\frac{v_{1y}}{J_1}ds\right\}
  -\exp\left\{-\gamma\int_0^t\frac{v_{2y}}{J_2}ds\right\}\right)\\
  &&+\left(\exp\left\{-\gamma\int_0^t\frac{v_{1y}}{J_1}ds\right\}
  -\exp\left\{-\gamma\int_0^t\frac{v_{2y}}{J_2}ds\right\}\right)
  \pi_0'.
  \end{eqnarray*}
  It follows from Proposition \ref{PropA2} and Proposition \ref{PropA3} that
  \begin{eqnarray*}
  &&\|\partial_y(\mathscr R_1(v_1)-\mathscr R_1(v_2))\|_{L^2(Q_T)}\\
  &\leq&\left\|\exp\left\{-\gamma\int_0^t\frac{v_{1y}}{J_1}ds\right\}
  -\exp\left\{-\gamma\int_0^t\frac{v_{2y}}{J_2}ds\right\}\right\|_{L^2(0,T;L^\infty)}
  \|\pi_0'\|_2\\
  &&+\|\pi_0\|_\infty\left\|\partial_y\left(\exp\left\{-\gamma\int_0^t\frac{v_{1y}}{J_1}ds\right\}
  -\exp\left\{-\gamma\int_0^t\frac{v_{2y}}{J_2}ds\right\}
\right)\right\|_{L^2(Q_T)}\\
  &\leq&T^{\frac12}\left\|
\exp\left\{-\gamma\int_0^t\frac{v_{1y}}{J_1}ds\right\}
  -\exp\left\{-\gamma\int_0^t\frac{v_{2y}}{J_2}ds\right\}\right\|_{L^\infty(Q_T)}
  \|\pi_0'\|_2 \\
  &&+\|\pi_0\|_\infty T^{\frac12}\left\|\partial_y\left(\exp\left\{-\gamma\int_0^t\frac{v_{1y}}{J_1}ds\right\}
  -\exp\left\{-\gamma\int_0^t\frac{v_{2y}}{J_2}ds\right\}\right)\right\|_{L^\infty(0,T;L^2)}\\
  &\leq& C(\|\pi_0'\|_2T^{\frac54}+\|\pi_0\|_\infty T^{\frac12})\|\delta v_y\|_{V_T},
  \end{eqnarray*}
  proving the conclusion.
\end{proof}

\begin{proposition}
  \label{PropA5}
  It holds that
  $$
  \|\partial_y(\mathscr R_2(v_1)-\mathscr R_2(v_2))\|_{L^2(Q_T)}\leq CT^{\frac12}\|\partial_y(v_1-v_2)\|_{V_T},
  $$
 for any $v_1, v_2\in\mathscr K_{M,T}$, and for a positive constant $C$ depending only on $\gamma, L, \underline J,$ and $ \|J_0'\|_2$.
\end{proposition}

\begin{proof}
    For simplicity of notations, for $v_1, v_2\in\mathscr K_{M,T}$, we denote $\delta v=v_1-v_2$, $J_i=\mathscr Q(v_i)$, $i=1,2$, and $\delta J=J_1-J_2$. Straightforward calculations yield
  \begin{eqnarray*}
    &&\partial_y(\mathscr R_2(v_1)-\mathscr R_2(v_2))\\
    &=&\int_0^te^{-\gamma\int_\tau^t\frac{v_{1y}}{J_1}ds}\Big(\frac{v_{1y}}{J_1}
    +\frac{v_{2y}}{J_2}\Big)\Big(\frac{v_{1y}}{J_1}
    -\frac{v_{2y}}{J_2}\Big)_yd\tau\\
    &&+\int_0^te^{-\gamma\int_\tau^t\frac{v_{1y}}{J_1}ds}\Big(\frac{v_{1y}}{J_1}
    +\frac{v_{2y}}{J_2}\Big)_y\Big(\frac{v_{1y}}{J_1}
    -\frac{v_{2y}}{J_2}\Big)d\tau\\
    &&-\gamma \int_0^te^{-\gamma\int_\tau^t\frac{v_{1y}}{J_1}ds}\int_\tau^t\Big(\frac{v_{1y}}{J_1}\Big)_yds\Big(\frac{v_{1y}}{J_1}
    +\frac{v_{2y}}{J_2}\Big)\Big(\frac{v_{1y}}{J_1}
    -\frac{v_{2y}}{J_2}\Big)d\tau\\
    &&+\int_0^t\partial_y\left(e^{-\gamma\int_\tau^t\frac{v_{1y}}{J_1}ds}-e^{-\gamma\int_\tau^t\frac{v_{2y}}{J_2}ds}\right)
    \Big(\frac{v_{2y}}{J_2}\Big)^2ds\\
    &&+2\int_0^t\left(e^{-\gamma\int_\tau^t\frac{v_{1y}}{J_1}ds}-e^{-\gamma\int_\tau^t\frac{v_{2y}}{J_2}ds}\right)
    \frac{v_{2y}}{J_2}\Big(\frac{v_{2y}}{J_2}\Big)_yds\\
    &=:&I_1+I_2+I_3+I_4+I_5.
  \end{eqnarray*}
  Estimates for $I_i, i=1, 2, 3, 4, 5$, are give as follows.
  By Proposition \ref{PropA0} and Proposition \ref{PropA3}
  \begin{eqnarray*}
    \|I_1\|_2&\leq&e^{\gamma\|\frac{v_{1y}}{J_1}\|_{L^1(0,T;L^\infty)}}\int_0^T\Big(\Big\|\frac{v_{1y}}{J_1}\Big\|_\infty
    +\|\frac{v_{2y}}{J_2}\Big\|_\infty\Big)\Big\|\Big(\frac{v_{1y}}{J_1}-\frac{v_{2y}}{J_2}\Big)_y\Big\|_2d\tau\\
    &\leq&e^{\frac{2\gamma C_1}{\underline J}}\Big(\Big\|\frac{v_{1y}}{J_1}\Big\|_{L^2(0,T;L^\infty)}
    +\Big\|\frac{v_{2y}}{J_2}\Big\|_{L^2(0,T;L^\infty)}\Big)
    \Big\|\Big(\frac{v_{1y}}{J_1}-\frac{v_{2y}}{J_2}\Big)_y\Big\|_{L^2(Q_T)}\\
    &\leq& C\|\delta v_y\|_{V_T}.
  \end{eqnarray*}
  Similarly, it follows from Propositions \ref{PropA0}, \ref{PropA2}, and \ref{PropA3} that
  \begin{eqnarray*}
    \|I_2\|_2&\leq&e^{\frac{2\gamma C_1}{\underline J}}\Big(\Big\|\Big(\frac{v_{1y}}{J_1}\Big)_y\Big\|_{L^2(Q_T)}
    +\Big\|\Big(\frac{v_{2y}}{J_2}\Big)_y\Big\|_{L^2(Q_T)}\Big)
    \Big\| \frac{v_{1y}}{J_1}-\frac{v_{2y}}{J_2} \Big\|_{L^2(0,T;L^\infty)}\\
    &\leq& Ce^{\frac{2\gamma C_1}{\underline J}}(1+M+\|J_0'\|_2)T^{\frac14}\|\delta v_y\|_{V_T}\leq C\|\delta v_y\|_{V_T},
  \end{eqnarray*}
  where we have used (\ref{H}). It follows from Propositions \ref{PropA0}--\ref{PropA3} that
  \begin{eqnarray*}
    \|I_3\|_2
    &\leq&\gamma e^{\frac{2\gamma C_1}{\underline J}}\int_0^T
    \left(\int_0^T\Big\|\Big(\frac{v_{1y}}{J_1}\Big)_y\Big\|_2 ds\right)\Big\|\frac{v_{1y}}{J_1}+\frac{v_{2y}}{J_2}\Big\|_\infty
    \Big\|\frac{v_{1y}}{J_1}-\frac{v_{2y}}{J_2}\Big\|_\infty d\tau\\
    &\leq&C\Big\|\Big(\frac{v_{1y}}{J_1}\Big)_y\Big\|_{L^1(0,T;L^2)}
    \Big\|\frac{v_{1y}}{J_1}+\frac{v_{2y}}{J_2}\Big\|_{L^2(0,T;L^\infty)}
    \Big\|\frac{v_{1y}}{J_1}-\frac{v_{2y}}{J_2}\Big\|_{L^2(0,T;L^\infty)}\\
    &\leq&C(1+M+\|J_0'\|_2)T^{\frac14}\|\delta v_y\|_{V_T}\leq C\|\delta v_y\|_{V_T},
  \end{eqnarray*}
  where (\ref{H}) has been used. By Propositions \ref{PropA0}--\ref{PropA3}, one deduces
  \begin{eqnarray*}
    \|I_4\|_2&\leq&\int_0^t\Big\| \partial_y\left(e^{-\gamma\int_\tau^t\frac{v_{1y}}{J_1}ds}-e^{-\gamma\int_\tau^t\frac{v_{2y}}{J_2}ds}\right)
   \Big\|_2\Big\|\frac{v_{2y}}{J_2}\Big\|_\infty^2ds\\
   &\leq&C(1+\|J_0'\|_2)T^{\frac12}\|\delta v_y\|_{V_T}\int_0^T\Big\|\frac{v_{2y}}{J_2}\Big\|_\infty^2ds
   \leq C\|\delta v_y\|_{V_T},
  \end{eqnarray*}
  and
  \begin{eqnarray*}
    \|I_5\|_2&\leq&2\int_0^t\Big\|e^{-\gamma\int_\tau^t\frac{v_{1y}}{J_1}ds}-e^{-\gamma\int_\tau^t\frac{v_{2y}}{J_2}ds}
   \Big\|_\infty\Big\|\frac{v_{2y}}{J_2}\Big\|_\infty \Big\|\Big(\frac{v_{2y}}{J_2}\Big)_y\Big\|_2ds\\
   &\leq&CT^{\frac34}\|\delta v_y\|_{V_T} \Big\|\frac{v_{2y}}{J_2}\Big\|_{L^2(0,T;L^\infty)} \Big\|\Big(\frac{v_{2y}}{J_2}\Big)_y\Big\|_{L^2(Q_T)}\\
   &\leq&CT^{\frac34}(M+\|J_0'\|_2)\|\delta v_y\|_{V_T} \leq C\|\delta v_y\|_{V_T},
  \end{eqnarray*}
  where $MT^{\frac14}$ has been used. Therefore, we have
  \begin{eqnarray*}
    \|\partial_y(\mathscr R_2(v_1)-\mathscr R_2(v_2))\|_{L^2(Q_T)}\leq\sum_{i=1}^5\|I_i\|_{L^2(Q_T)}\leq CT^{\frac12}
    \|\delta v_y\|_{V_T},
  \end{eqnarray*}
  proving the conclusion.
\end{proof}

\begin{proposition}
  \label{PropA6}
  For any $v\in\mathscr K_{M,T}$, it holds that
  \begin{equation*}
    \|\partial_y\mathscr R(v)\|_{L^2(Q_T)}\leq C,
  \end{equation*}
  for a positive constant $C$ depending only on $\gamma, \mu, L, \underline J, \|J_0'\|_2, \|\pi_0\|_\infty,$ and $\|\pi_0'\|_2$.
\end{proposition}

\begin{proof}
  Note that $\mathscr R(0)=\pi_0$, it follows from Proposition \ref{PropA4} and Proposition \ref{PropA5} that
  \begin{eqnarray*}
\|\partial_y\mathscr R(v)\|_{L^2(Q_T)}&\leq&\|\partial_y\mathscr R(0)\|_{L^2(Q_T)}+\|\partial_y(\mathscr R(v)-\mathscr R(0))\|_{L^2(Q_T)}\\
    &\leq&\|\pi_0'\|_{L^2(Q_T)}+CT^{\frac12}\|v\|_{V_T}\leq T^{\frac12}\|\pi_0'\|_2+CMT^{\frac12}\leq C,
  \end{eqnarray*}
  where (\ref{H}) has been used. This proves the conclusion.
\end{proof}

\subsection{Properties of $\mathscr F$}
\begin{proposition}
  \label{PropA7}
  For any $v\in\mathscr K_{M,T}$, it holds that
  $$
  \|\partial_y\mathscr F(v)\|_{V_T}\leq C,
  $$
  for a positive constant $C$ depending only on $\gamma, \mu, L, \underline\varrho, \bar\varrho,
  \underline J, \|J_0'\|_2, \|\pi_0\|_\infty, \|\pi_0'\|_2$, and $\|v_0'\|_2$.
\end{proposition}

\begin{proof}
  Denote $J=\mathscr Q(v), \pi=\mathscr R(v)$, and
  $V=\mathscr F(v)$. Testing (\ref{L6}) with $-\frac{V_{yy}}{\varrho_0}$ and noticing
  $\frac{\underline J}{2}\leq J\leq 2\bar J$,
  one deduces
  \begin{eqnarray*}
    \frac{d}{dt}\|V_y\|_2^2+\mu\left\|\frac{V_{yy}}{\sqrt{J\varrho_0}}\right\|_2^2&=&\int_0^L\left(\frac{\pi_y}{\varrho_0}
    +\mu\frac{J_yv_y}{J\varrho_0}\right)V_{yy}dy \\
    &\leq& \varepsilon\|V_{yy}\|_2^2+C_\varepsilon(\|\pi_y\|_2^2+\|v_y\|_\infty^2\|J_y\|_2^2),
  \end{eqnarray*}
  for any positive $\varepsilon$, which, choosing $\varepsilon$ sufficiently small and applying Propositions \ref{PropA0}, \ref{PropA1}, and \ref{PropA6},
  gives
  \begin{eqnarray*}
    \sup_{0\leq t\leq T}\|V_y\|_2^2+\int_0^T\|V_{yy}\|_2^2dt
    &\leq&  C(\|v_0'\|_2^2+\|\pi_y\|_{L^2(Q_T)}^2+\|J_y\|_{L^\infty(0,T;L^2)}^2\|v_y\|_{L^2(0,T;L^\infty)}^2)\\
    &\leq&  C,
  \end{eqnarray*}
  proving the conclusion.
\end{proof}

\begin{proposition}
  \label{PropA8}
  It holds that
  $$
  \|\partial_y(\mathscr F(v_1)-\mathscr F(v_2))\|_{V_T}\leq CT^{\frac14}\|\partial_y(v_1-v_2)\|_{V_T},\quad\forall v_1, v_2\in\mathscr K_{M,T},
  $$
  for a positive constant $C$ depending only on $\gamma, \mu, L, \underline\varrho, \bar\varrho,
  \underline J, \|J_0'\|_2, \|\pi_0\|_\infty, \|\pi_0'\|_2$, and $\|v_0'\|_2$.
\end{proposition}

\begin{proof}
  Denote $J_i=\mathscr Q(v_i), \pi_i=\mathscr R(v_i), V_i=\mathscr Fv_i, i=1,2.$  Set $\delta J=J_1-J_2, \delta\pi=\pi_1-\pi_2,$ and $\delta V=V_1-V_2$. Then,
  \begin{equation*}
    \delta V_t-\frac\mu{\varrho_0J_1}\delta V_{yy}=-\frac{V_{2yy}}{J_1J_2\varrho_0}\delta J-\Big[\frac{\delta\pi_y}{\varrho_0}
    +\frac\mu{\varrho_0}\Big(\frac{J_{1y}}{J_1^2}\delta v_y+\frac{v_{2y}}{J_1^2}\delta J_y-\frac{J_1+J_2}{J_1^2J_2^2}\delta J
    J_{2y}v_{2y}\Big)\Big].
  \end{equation*}
  Testing the above with $-\delta V_{yy}$ and using Proposition \ref{PropA1}, one deduces
  \begin{eqnarray*}
    &&\frac12\frac{d}{dt}\|\delta V_y\|_2^2+\frac{\mu}{2\bar\varrho\bar J}\|\delta V_{yy}\|_2^2\\
    &\leq&\frac\mu{4\bar\varrho\bar J}\|\delta V_{yy}\|_2^2+C(\|V_{2yy}\|_2^2\|\delta J\|_\infty^2+\|\delta\pi_y\|_2^2+\|J_{1y}\|_2^2\|\delta v_y\|_\infty^2\\
    &&+\|v_{2y}\|_\infty^2\|\delta J_y\|_2^2+\|J_{2y}\|_2^2\|v_{2y}\|_\infty^2\|\delta J\|_\infty^2),
  \end{eqnarray*}
  which, integrating with respect to $t$, and applying Propositions \ref{PropA0}, \ref{PropA1}, \ref{PropA4}, \ref{PropA5}, and \ref{PropA7}, yields
  \begin{eqnarray*}
    \|\delta V_y\|_{V_T}^2&=&\sup_{0\leq t\leq t}\|\delta V_y\|_2^2+\int_0^T\|\delta V_{yy}\|_2^2dt\\
    &\leq& C(\|V_{2yy}\|_{L^2(Q_T)}^2\|\delta J\|_{L^\infty(Q_T)}^2+\|\delta\pi_y\|_{L^2(Q_T)}^2\\
    &&+\|J_{1y}\|_{L^\infty(0,T;L^2)}^2
    \|\delta v_y\|_{L^2(0,T;L^\infty)}^2+\|v_{2y}\|_{L^2(0,T;L^\infty)}^2\|\delta J_y\|_{L^\infty(0,T;L^2)}^2\\
    &&+\|J_{2y}\|_{L^\infty(0,T;L^2)}^2\|v_{2y}\|_{L^2(0,T;L^\infty)}^2\|\delta J\|_{L^\infty(Q_T)}^2)\\
    &\leq&CT^{\frac12}\|\delta v_y\|_{V_T}^2,
  \end{eqnarray*}
  proving the conclusion.
\end{proof}

\begin{corollary}
  \label{CorLoc}
  There is a positive constant $C_\#$ depending only on $\gamma, \mu, L, \underline\varrho$, $\bar\varrho$,
  $\underline J, \|J_0'\|_2, \|\pi_0\|_\infty, \|\pi_0'\|_2$, and $\|v_0'\|_2$, such that for any $M\geq C_\#$, it follows
  $$
  \|\partial_y\mathscr F(v)\|_{V_{T_\#}}\leq M,\quad\|\partial_y(\mathscr F(v_1)-\mathscr F(v_2))\|_{V_{T_\#}}\leq\frac12
  \|\partial_y(v_1-v_2)\|_{V_{T_\#}},
  $$
  for any $v, v_1, v_2\in\mathscr K_{M, T_\#}$, where
  $$
  T_\#:=\min\left\{\frac{1}{M^4}, \frac{1}{16C_\#^4}, 1, \left(\frac{\underline J}{2C_1}\right)^2\right\}.
  $$
\end{corollary}

\begin{proof}
By Proposition \ref{PropA7} and Proposition \ref{PropA8}, there is a positive constant $C_\#$ depending only on $\gamma, \mu, L, \underline\varrho, \bar\varrho,
  \underline J, \|J_0'\|_2, \|\pi_0\|_\infty, \|\pi_0'\|_2$, and $\|v_0'\|_2$, such that
 \begin{equation}
 \label{L8}
  \|\partial_y\mathscr F(v)\|_{V_{T}}\leq C_\#,\quad\|\partial_y(\mathscr F(v_1)-\mathscr F(v_2))\|_{V_{T}}\leq C_\#T^{\frac14}
  \|v_1-v_2\|_{V_{T}},
  \end{equation}
  for any $v, v_1, v_2\in\mathscr K_{M,T}$, for any $M,T$ satisfying
  $$
  MT^{\frac14}\leq1, \quad T\leq1, \quad T\leq\left(\frac{\underline J}{2C_1}\right)^2.
  $$
  For any $M\geq C_\#$, choose
  $$
  T_\#:=\min\left\{\frac{1}{M^4}, \frac{1}{16C_\#^4}, 1, \left(\frac{\underline J}{2C_1}\right)^2\right\}.
  $$
  Then, by (\ref{L8}), one has
  $$
  \|\partial_y\mathscr F(v)\|_{V_{T_\#}}\leq M,\quad\|\partial_y(\mathscr F(v_1)-\mathscr F(v_2))\|_{V_{T_\#}}\leq \frac12
  \|v_1-v_2\|_{V_{T_\#}},
  $$
  for any $v, v_1, v_2\in\mathscr K_{M, T_\#}$, proving the conclusion.
\end{proof}

\subsection{Properties of $\mathscr F$ and the local well-posedness}

\begin{proof}[\textbf{Proof of Proposition \ref{PropLocal}}]
Let $C_\#$ be the positive constant in Corollary \ref{CorLoc}. Set $M=C_\#$ and let $T_\#$ be the corresponding
positive time in Corollary \ref{CorLoc}. Recall the definition of $\mathscr K_{M_\#,T_\#}$ and define
$|||v|||:=\|v_y\|_{V_T}$, for any $v\in \mathscr K_{M_\#,T_\#}$. By the Poincar\'e inequality, one can easily
check that $|||\cdot|||$ is a norm on the space $X_{T_\#}$ and is equivalent to the $L^\infty(0,T_\#;H_0^1)\cap L^2(0,T_\#; H^2)$ norm. Consequently, $\mathscr K_{M_\#,T_\#}$ is a completed metric space, equipped with the
metric $d(v_1,v_2):=|||v_1-v_2|||=\|\partial_y(v_1-v_2)\|_{V_T}$. Let $\mathscr Q, \mathscr R, \mathscr F$ be the mappings defined as before.
By Corollary \ref{CorLoc}, $\mathscr F$ is a contractive mapping
on $\mathscr K_{M_\#,T_\#}$. Therefore, by the
contractive mapping principle, there is a unique fixed point, denoted by $v_\#$, to $\mathscr F$ on $\mathscr K_{M_\#,T_\#}$.
Set $J_\#=\mathscr Q(v_\#)$ and $\pi_\#=\mathscr R(v_\#)$. By the definitions of $\mathscr Q(v_\#)$ and $\mathscr R(v_\#)$, one can easily check that $(J_\#, v_\#, \pi_\#)$ is a solution to system (\ref{EQJ})--(\ref{EQpi}), subject to (\ref{BC})--(\ref{IC}). The regularities of $(J_\#, v_\#, \pi_\#)$ can be verified through straightforward computations
to the expressions of $\mathscr Q(v)$ and $\mathscr R(v)$ and using (\ref{L6}). Since the calculations are standard, we omit
the details here.
\end{proof}

\section*{Acknowledgments}
This work was supported in part by the National Natural Science Foundation of China NSFC 11971009, NSFC 11771156, and NSFC 11871005, the start-up grant of the South China Normal University 550-8S0315, and the Hong Kong RGC grant CUHK 14302917.

\end{document}